\documentclass[10pt, reqno]{amsart}
\usepackage{graphicx, amssymb, amsmath, amsthm}
\numberwithin{equation}{section}
\usepackage{xcolor}
\usepackage{amscd}
\usepackage{amssymb}
\usepackage{latexsym}
\usepackage{tikz}
\usetikzlibrary{arrows.meta,calc}

\usepackage{amsmath}
\usepackage{amscd}
\usepackage{cite}
\usepackage{color}
\usepackage{enumerate}
\usepackage{amsfonts}
\usepackage{graphicx}
\usepackage{mathrsfs}
\usepackage{setspace}
\usepackage{hyperref}

\let\Re=\undefined\DeclareMathOperator*{\Re}{Re}
\let\Im=\undefined\DeclareMathOperator*{\Im}{Im}

\newcommand{\qtq}[1]{\quad\text{#1}\quad}
\newcommand{\R}{\mathbb{R}}
\newcommand{\C}{\mathbb{C}}

\newcommand{\eps}{\varepsilon}

\newcommand{\F}{\mathcal{F}}

\newtheorem{theorem}{Theorem}[section]

\newtheorem{lemma}[theorem]{Lemma}

\newtheorem{corollary}[theorem]{Corollary}

\newtheorem{proposition}[theorem]{Proposition}

\theoremstyle{definition}
\newtheorem{definition}[theorem]{Definition}
\newtheorem{remark}[theorem]{Remark}

\title[Threshold and blow-threshold dynamics]{Threshold and Below-Threshold Dynamics for the Energy-Critical NLS: The Higher-Kinetic-Energy Regime}

	\author[Z. Ma]{Zuyu Ma}
	\address[Z. Ma]{  The Graduate School of China Academy of Engineering Physics,
		Beijing 100088,\ P. R. China}
	\email{mazuyu23@gscaep.ac.cn}

\author[C. Miao]{Changxing Miao}
\address[C. Miao]{School of Mathematics and Physics, University of Science and Technology Beijing, Beijing 100083, China}
\email{miao\_changxing@ustb.edu.cn, miao\_changxing@iapcm.ac.cn}

\author[K. Yang]{Kai Yang}
\address[K. Yang]{School of Mathematics, SouthEast University, Nanjing, P.R. China 211189}
\email{kaiyang@seu.edu.cn}

\author[X. Zhang]{Xiaoyi Zhang}
\address[X. Zhang]{Department of Mathematics, University of Iowa, Iowa City, IA 52242}
\email{xiaoyi-zhang@uiowa.edu}

\author[J. Zheng]{Jiqiang Zheng}
\address[J. Zheng]{Institute of Applied Physics and Computational Mathematics and National Key Laboratory of Computational Physics, Beijing 100088, China}
\email{zheng\_jiqiang@iapcm.ac.cn, zhengjiqiang@gmail.com}

\begin{document}
\begin{abstract}
{
We study the focusing energy-critical nonlinear Schr\"odinger equation below and at the ground-state energy, focusing on solutions in the upper energy cone, characterized by
\[
E(u)\le E(W),\qquad
\|\nabla u\|_{L_x^2}>\|\nabla W\|_{L_x^2},
\]
where \(W\) is the ground state. Unlike the lower cone, where compactness of critical solutions plays a central role in understanding the dynamics, such compactness is generally unavailable on the upper branch. Consequently, previous results have often required additional assumptions, such as finite mass or finite variance. In this paper, we remove the finite-mass assumption in the radial setting and give a complete characterization of the upper-cone dynamics in the natural energy space \(\dot H^1_{\mathrm{rad}}\), resolving an open problem posed by Duyckaerts and Merle\cite{DM08}. In particular, this establishes finite-time blow-up of the special threshold solution \(W^+\) in the negative time direction in dimensions \(d=3,4\). For non-radial data in dimensions \(d\ge5\), we show that {for solutions in the interior of the upper cone, the limit superior of their $\dot H^1$-norm satisfies an explicit lower bound as time approaches each endpoint of the maximal lifespan,} while forward-global solutions on the upper threshold sheet that remain below this level must coincide with \(W^+\) modulo symmetries. We also obtain a complete classification of upper-threshold dynamics for finite-variance data in all dimensions \(d\ge3\).}
\end{abstract}

\maketitle

\section{Introduction}

We consider the focusing energy-critical nonlinear Schr\"odinger equation 
\begin{align} \label{1.1}\tag{NLS}
\qquad \qquad (i\partial_t+\Delta)u= -|u|^{\frac4{d-2}} u,\quad (t,x)\in\R\times\R^d,
\end{align}
in dimensions $d\geq 3$. Equation \eqref{1.1} is the Hamiltonian flow associated with the (conserved) energy
\begin{equation}\label{equ:energydef}
E(u(t))=\tfrac12\int_{\R^d}|\nabla u(t,x)|^2\;dx-\tfrac{d-2}{2d}\int_{\R^d}|u(t,x)|^{\frac{2d}{d-2}}\;dx.
\end{equation}

For initial data $u|_{t=0}=u_0$, we understand a solution $u(t,x)$ to be a maximal-lifespan mild solution on the interval 
$$
I=(T_-,T_+),\qquad -\infty\leq T_-<0<T_+\leq\infty,
$$ 
such that 
$$
u\in C(I; \dot H^1)\cap L_{t_{\textit{loc}},x}^{\frac{2(d+2)}{d-2}}(I\times \mathbb R^d)
$$  
and $u$ satisfies the usual Duhamel formula. For a time interval $J\subset I$, we define the scattering size by
\[
S_J(u)=\|u\|_{L_{t,x}^{\frac{2(d+2)}{d-2}}(J\times\mathbb R^d)}^{\frac{2(d+2)}{d-2}}.
\]
We say $u$ blows up at $T_{\pm}$ if $S_{(0,T_+)}(u)=\infty$ (or $S_{(T_-,0)}(u)=\infty$). For a solution that is global forward or backward in time, blow-up means non-scattering in that time direction. 

A central role in the dynamics of \eqref{1.1} is played by the \emph{ground state}, which is the optimizer of the sharp Sobolev inequality and is given explicitly by 
\begin{equation}
\label{defW}
W(x):= \biggl(1+\tfrac{|x|^2}{d(d-2)}\biggr)^{-\frac{d-2}{2}}.
\end{equation}
The ground state satisfies the Euler-Lagrange equation
\begin{equation}\label{equ:CP}
\Delta W+|W|^\frac{4}{d-2}W=0.
\end{equation}
By the phase, scaling, and spatial translation symmetries of \eqref{1.1}, $W$ generates the $(d+2)$-dimensional ground-state orbit $\mathcal W$:
\begin{equation}\label{groundm}
\mathcal W=\ \bigl\{ e^{i\theta}\lambda^{-\frac{d-2}2}W\biggl(\frac{x-x_0}\lambda\biggr)
,\ (\theta,\lambda,x_0)\in \mathbb R\times\mathbb R^+\times\mathbb R^d\bigr\}. 
\end{equation}

We consider solutions with energy at or below the ground-state energy, $E(u)\le E(W)$. Since $W$ is a critical point of $E$ and, modulo $(d+2)$ symmetry directions, $E''(W)$ has one negative direction and is positive in the remaining transverse directions, the threshold energy surface $E(u)=E(W)$ has a local double-cone structure near the ground-state orbit $\mathcal W$. Motivated by this local geometry, and in order to describe the dynamics below and at the threshold, we regard $\{u: E(u)=E(W)\}$ as a double cone whose vertex set is $\mathcal W$. The upper sheet is characterized by 
$$\|\nabla u\|_{L_x^2}>\|\nabla W\|_{L_x^2},
$$ 
whereas the lower sheet is characterized by 
$$
\|\nabla u\|_{L_x^2}<\|\nabla W\|_{L_x^2}.
$$
We use the same terminology below the threshold: the interior of the upper cone consists of functions satisfying
$$
E(u)<E(W),\qquad  \|\nabla u\|_2>\|\nabla W\|_2,
$$ 
while the interior of the lower cone is characterized by 
$$
E(u)<E(W), \qquad \|\nabla u\|_2<\|\nabla W\|_2. 
$$
By the variational characterization of the ground state, together with conservation of energy and continuity of \(t\mapsto \|\nabla u(t)\|_2\), the upper and lower sheets of the threshold surface, as well as the interiors of the upper and lower cones, are invariant under the \eqref{1.1} flow. More precisely, if \(u(t_0)\) belongs to one of these regions at some time \(t_0\), then \(u(t)\) remains in the same region throughout its lifespan.

Figure 1 illustrates this double-cone terminology.

\begin{figure}[htbp]
\centering
\resizebox{0.78\textwidth}{!}{%
\begin{tikzpicture}[>=Latex,font=\small]

\definecolor{upperblue}{RGB}{120,160,220}
\definecolor{lowerred}{RGB}{220,140,140}

\coordinate (O) at (0,0);

\def\topY{3.6}
\def\topA{3.2}
\def\topB{0.55}

\def\botY{-3.6}
\def\botA{3.2}
\def\botB{0.55}

\draw[thick,->] (0,-5.0) -- (0,5.0);
\node[above] at (0,5.05) {kinetic energy};

\fill[upperblue!25]
  (-\topA,\topY) --
  (0,0) --
  (\topA,\topY) --
  plot[domain=0:180,smooth] ({\topA*cos(\x)},{\topY+\topB*sin(\x)})
  -- cycle;

\draw[thick,blue!70!black]
  (-\topA,\topY) -- (0,0) -- (\topA,\topY);

\draw[thick,blue!70!black] (0,\topY) ellipse ({\topA} and {\topB});
\draw[dashed,blue!50!black] (0,\topY) ellipse ({0.82*\topA} and {0.42*\topB});

\fill[lowerred!25]
  (-\botA,\botY) --
  (0,0) --
  (\botA,\botY) --
  plot[domain=0:-180,smooth] ({\botA*cos(\x)},{\botY+\botB*sin(\x)})
  -- cycle;

\draw[thick,red!75!black]
  (-\botA,\botY) -- (0,0) -- (\botA,\botY);

\draw[thick,red!75!black] (0,\botY) ellipse ({\botA} and {\botB});
\draw[dashed,red!55!black] (0,\botY) ellipse ({0.82*\botA} and {0.42*\botB});

\draw[thick,red!80!black,fill=white] (0,0) ellipse (0.28 and 0.08);

\node at (0,2.0) {upper interior cone};
\node at (0,-2.1) {lower interior cone};

\draw[->,thick] (2.25,1.75) -- (1.55,2.15);
\node[anchor=west,align=left] at (2.3,1.75) {upper threshold\\sheet / surface};

\draw[->,thick] (2.25,-1.95) -- (1.55,-2.35);
\node[anchor=west,align=left] at (2.3,-1.95) {lower threshold\\sheet / surface};

\draw[->,thick] (1.8,0.15) -- (0.35,0.03);
\node[anchor=west] at (1.85,0.15) {ground-state manifold $\mathcal W$};

\draw[dashed] (-2.35,0) -- (-0.38,0);
\node[anchor=east,align=center] at (-2.5,0)
{ground-state\\kinetic level\\[1mm]
$\displaystyle E_{\mathrm{kin}}=E_0$};

\draw[thick,<->] (-3.3,1.0) -- (-3.3,2.8);
\node[anchor=east,align=center] at (-3.55,1.9)
{higher kinetic energy\\[1mm]
$\displaystyle E_{\mathrm{kin}}>E_0$};

\draw[thick,<->] (-3.3,-1.0) -- (-3.3,-2.8);
\node[anchor=east,align=center] at (-3.55,-1.9)
{lower kinetic energy\\[1mm]
$\displaystyle E_{\mathrm{kin}}<E_0$};

\end{tikzpicture}
}
\caption{\small{Schematic of the double-cone terminology: $E_{kin}(u)= \|\nabla u\|_{L_x^2}^2$ and $E_0= \|\nabla W\|_{L_x^2}^2$. The cone boundaries denote $E(u)=E(W)$; the shaded interiors denote $E(u)<E(W)$.}} 
\label{fig:doublecone}
\end{figure}
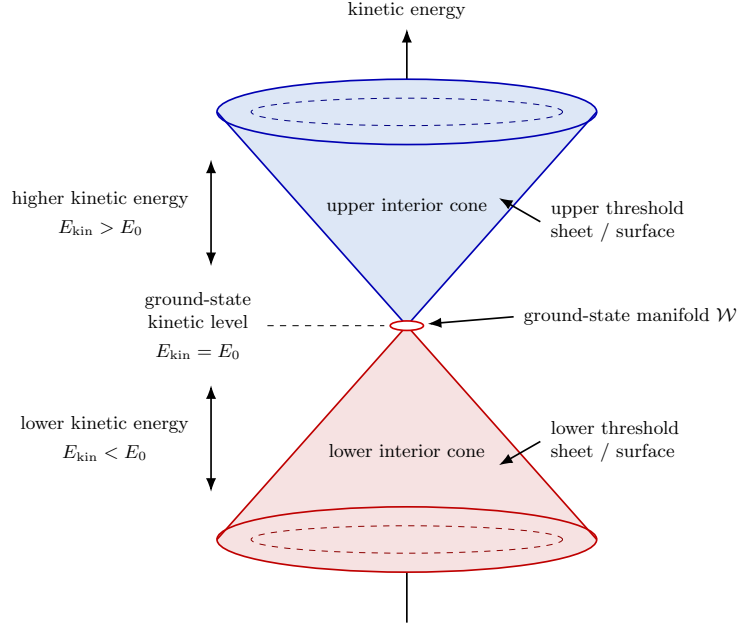

The dynamics in the lower cone are now well understood. For radial solutions in dimensions $d=3,4,5$, Kenig and Merle~\cite{kenig-merle1} proved that the data in the interior of the lower cone generate global solutions that scatter in both time directions. Killip and
Vi\c{s}an~\cite{KV1} subsequently removed the radial assumption in $d\ge 5$, and Dodson~\cite{dodson1} established the corresponding result in dimension $d=4$. Very recently, {Su-Zhao~\cite{SZ} and Chung-Han ~\cite{CH} announced two different proofs of the remaining non-radial three-dimensional case. }Thus, subject to this recent result, the scattering theory in the interior of the lower cone is complete in every dimension $d\ge 3$.

On the threshold surface \(E(u)=E(W)\), the dynamics are more delicate because of the stable and unstable directions of the linearized flow around the ground-state orbit \(\mathcal W\). At \(W\), there is a one-dimensional stable manifold whose two branches are represented by two special solutions \(W^+\) and \(W^-\) up to time translation, and both solutions converge exponentially to \(W\) as \(t\to+\infty\). The solution \(W^+\) lies on the upper sheet of the energy cone, while \(W^-\) lies on the lower sheet. The corresponding unstable manifold is obtained by time reversal, and the stable and unstable manifolds associated with the other elements of \(\mathcal W\) are obtained by applying the \((d+2)\)-parameter symmetries of the equation.

The existence of \(W^\pm\) and the classification of threshold dynamics were initiated by Duyckaerts and Merle~\cite{DM08} and subsequently extended in~\cite{LZ09,CFR22,MMMZ25}. Thus, unlike the interior of the lower cone, where scattering gives the dominant behavior, the threshold surface contains special non-scattering trajectories attached to the ground-state orbit. These exceptional solutions play a central role in the dynamics on the upper sheet and will be particularly relevant to the results of this paper.

Our main interest is the upper branch of the energy cone. We first recall the known dynamics on the two branches, including the recent three-dimensional scattering result \cite{SZ,CH}.

For convenience, we say that ``$u=v$ modulo symmetries" if
\[
 u(t,x)=e^{i\theta}\lambda^{-\frac{d-2}2}
 v\left(\frac{t-T}{\lambda^2},\frac{x-y}{\lambda}\right),
 \]
 or
 \[
 u(t,x)=e^{i\theta}\lambda^{-\frac{d-2}2}
 \overline {v\left(\frac{T-t}{\lambda^2},\frac{x-y}{\lambda}\right)}
 \]
for some $(\lambda,\theta,y,T)\in(0,\infty)\times\R\times\R^d\times\R$.

\begin{theorem}[Known results on the lower branch, \cite{kenig-merle1, KV1, dodson1, SZ, CH, DM08, LZ09, CFR22, MMMZ25}] \label{thm:KM} Let $d\ge 3$ and let $u_0\in \dot H^1(\mathbb R^d)$ satisfy $$\|\nabla u_0\|_{L_x^2}<\|\nabla W\|_{L_x^2}.$$ Then

\begin{enumerate}
\item[(i)] If $E(u_0)<E(W)$, the corresponding solution is global and scatters in both time directions. 

\item[(ii)] If $E(u_0)=E(W)$, in dimension $d=3$ we further assume that $u_0$ is radial. Then the corresponding solution is global and either scatters in both time directions or satisfies $u=W^-$ modulo symmetries. 
\end{enumerate}
\end{theorem}

Compactness of non-scattering critical solutions plays a central role in the analysis of the lower branch. Through concentration-compactness, failure of scattering leads to a critical element whose trajectory is precompact modulo symmetries, and the minimality of $E(W)$ then enables a rigidity argument. On the upper branch, such compactness is no longer available in the same form. Consequently, classical characterizations of upper-branch dynamics require additional spatial localization, such as $H^1$ regularity or finite variance. 

We write 
\[
\Sigma:=\{f\in H^1(\mathbb R^d): |x|f \in L_x^2 (\mathbb R^d)\}.
\]
The previously known upper-branch results can be summarized as follows.

\begin{theorem}[Known results on the upper branch, \cite{kenig-merle1, KV1, DM08, LZ09, CFR22}] Let $d\ge 3$ and suppose 
$$\|\nabla u_0\|_{L_x^2}>\|\nabla W\|_{L_x^2}.$$ Then
\begin{enumerate}
\item[(i)] If $E(u_0)<E(W)$ and either $u_0\in H^1_{\mathrm{rad}}$ or $u_0\in \Sigma$, the corresponding solution blows up in finite time in both directions.
\item[(ii)] If $E(u_0)=E(W)$ and $u_0\in H^1_{\mathrm{rad}}$, then either $u$ blows up in finite time in both directions or $u=W^+$ modulo symmetries.
\end{enumerate} 
\end{theorem}

The $L^2$-based regularity or decay assumptions in these upper-branch results are not natural for this energy-critical problem whose natural phase space is $\dot H^1$. This is particularly apparent in dimensions $d=3,4$, where neither the ground-state $W$ nor the distinguished threshold solutions $W^\pm$ belong to $L^2$. Thus, some of the central objects governing the threshold dynamics are excluded from the framework of the previous results. This limitation was already recognized by Duyckaerts and Merle ~\cite{DM08}. They conjectured that their upper-threshold classification should remain valid without $L^2$ assumption and also expected that $W^+$ should blow up in finite negative time in $d=3,4$, as they had proved in dimension $d=5$. Both questions point to the need for an upper-branch theory formulated directly in the natural energy space $\dot H^1_{\mathrm{rad}}$. 

Our first contribution answers this question affirmatively and gives a complete characterization of the radial upper-cone dynamics in the natural energy space $\dot H^1_{rad}$. No finite-mass assumption is required. Our results show that, in the interior of the upper cone, every radial $\dot H^1$ solution blows up in finite time in both directions; on the upper threshold surface, every radial solution either blows up in finite time in both directions or coincides with $W^+$ modulo symmetries. {In particular, we obtain} finite-time blow-up of $W^+$ in the negative time direction in dimensions $d=3,4$.

\begin{theorem}[Radial solutions in the interior of the upper cone]
\label{thm:radial-subthreshold}
Let $d\geq3$, and let $u:I\times\R^d\to\C$ be the radial maximal-lifespan solution to \eqref{1.1} with initial data $u|_{t=0}=u_0\in \dot H^1_{\mathrm{rad}}(\mathbb R^d)$ satisfying
\begin{equation}\label{hyp:subthreshold}
 E(u_0)<E(W),\qquad
 \|\nabla u_0\|_{L_x^2}>\|\nabla W\|_{L_x^2}.
\end{equation}
Then the solution blows up
in finite time in both time directions.
\end{theorem}

\begin{theorem}[Radial solutions on the upper threshold surface]\label{thm:radial-threshold}
Let $d\geq3$, and let $u:I\times\R^d\to\C$ be the radial maximal-lifespan solution to \eqref{1.1} with initial data $u|_{t=0}=u_0\in \dot H^1_{\mathrm{rad}}(\mathbb R^d)$ satisfying
\begin{equation}\label{hyp:bounded}
 E(u_0)=E(W),\qquad
 \|\nabla u_0\|_{L_x^2}>\|\nabla W\|_{L_x^2}.
\end{equation}
Then either $|I|<\infty$ or $u=W^+$ modulo symmetries.
\end{theorem}

We next turn to non-radial data. Our results in this setting are twofold. First, in dimensions \(d\ge5\), without imposing any spatial localization, we obtain a quantitative rigidity property for solutions on the upper branch. {Namely, for solutions in the interior of the upper cone, the limit superior of their $\dot H^1$-norm satisfies an explicit lower bound as time approaches each endpoint of the maximal lifespan.} On the upper threshold surface, if a solution remains uniformly below this threshold throughout its forward lifespan, then it is global forward in time and coincides with $W^+$ modulo symmetries. The {threshold constant} arising from our argument is
\[
\kappa_d:=\Big(\frac d{d-2}\Big)^{\frac{d-2}{4}}.
\]

Second, for finite-variance data, we obtain a complete characterization of the upper-threshold dynamics in every dimension \(d\ge3\). Previously, the finite-variance assumption yielded finite-time blow-up in the interior of the upper cone through the classical Glassey argument, where the second derivative of the variance is uniformly negative. At the threshold energy, this uniform negativity is no longer available, and the same argument does not provide a classification. We prove that for \(u_0\in\Sigma\) on the upper threshold surface, the only alternative to finite-time blow-up is the special solution \(W^+\) modulo symmetries.

We describe these results in the following two theorems. 
We also include the classical finite-time blow-up statement for data with finite variance \cite{kenig-merle1}.

\begin{theorem}[Non-radial solutions in the interior of the upper cone]
\label{thm:general-subthreshold}
Let $d\geq3$, let $u_0\in\dot H^1(\R^d)$ satisfy
\eqref{hyp:subthreshold}, and let $u:I\times\R^d\to\C$ {, $I=(T_-,T_+)$,} be the
corresponding maximal-lifespan solution.
\begin{enumerate}
\item[(i)] If $u_0\in\Sigma$, then $|I|<\infty$.
\item[(ii)] If $d\geq5$, then
\[
{ \limsup_{t\uparrow T_+}\|\nabla u(t)\|_{L_x^2}
 \geq\kappa_d\|\nabla W\|_{L_x^2},
 \qquad
 \limsup_{t\downarrow T_-}\|\nabla u(t)\|_{L_x^2}
 \geq\kappa_d\|\nabla W\|_{L_x^2}.
}\]
\end{enumerate}
\end{theorem}

\begin{theorem}[Non-radial solutions on the upper threshold surface]
\label{thm:general-threshold}
Let $d\geq3$, let $u_0\in\dot H^1(\R^d)$ satisfy
\eqref{hyp:bounded}, and let $u:I\times\R^d\to\C$ be the
corresponding maximal-lifespan solution.
\begin{enumerate}
\item[(i)] Suppose further that $u_0\in\Sigma$. Then in dimensions $3\le d\le 6$, $|I|<\infty$. In dimensions $d\ge 7$, either $|I|<\infty$ or
$u=W^+$ modulo symmetries.

\item[(ii)] Suppose $d\geq5$, $I=(T_-,T_+)$ and
\begin{equation}\label{hyp:ceiling}
 M:=\sup_{t\in[0,T_+)}\|\nabla u(t)\|_{L_x^2}
 <\kappa_d\|\nabla W\|_{L_x^2}.
\end{equation}
Then  $T_+=\infty$ and $u=W^+$ modulo a fixed phase,
scaling, spatial translation and time translation.
{In particular, since $W^+$ has finite backward lifespan, no
maximal-lifespan solution satisfying \eqref{hyp:bounded} can obey both}
\[
{ \limsup_{t\uparrow T_+}\|\nabla u(t)\|_{L_x^2}
 <\kappa_d\|\nabla W\|_{L_x^2},
 \qquad
 \limsup_{t\downarrow T_-}\|\nabla u(t)\|_{L_x^2}
 <\kappa_d\|\nabla W\|_{L_x^2}.
}\]
\end{enumerate}
\end{theorem}
\begin{remark}

    First, we note that the distinction between the dimension ranges in (i) arises from the fact that $W^+\in \Sigma$ iff $d\geq7$ (as shown in Appendix \ref{Wplus}). In other words, $W^+$ is excluded from $\Sigma$ in low dimensions, so every finite-variance solution in these dimensions blows up in finite time in both directions.

     Second, the constant $\kappa_d$ in (ii) comes from the coercivity estimate in Lemma~\ref{lem:ceiling-coercivity}. In ongoing work, we obtain a larger admissible kinetic-energy region than in \eqref{hyp:ceiling}; this refinement will be addressed elsewhere.  
\end{remark}

We next outline the proofs. To this end, we define
\begin{equation}\label{def:deltaK}
 \delta(f):=\|\nabla f\|_{L_x^2}^2-\|\nabla W\|_{L_x^2}^2,
 \
 K(f):=\|\nabla f\|_{L_x^2}^2-\|f\|_{L_x^{2^*}}^{2^*},
\end{equation}
where $2^*=\frac{2d}{d-2}.$

\subsection*{Outline of the proof}
 
The proofs for the interior of the upper cone and for the upper threshold surface are  closely related. In the radial setting, the first step is to rule out global solutions for which the virial functional remains uniformly negative. Below the ground-state energy, the variational estimates provide such a uniform gap, which yields Theorem~\ref{thm:radial-subthreshold}. At the threshold, the same argument gives convergence to the ground-state manifold along a sequence of times. A no-return argument then shows that a forward-global solution must eventually remain near \(\mathcal W\), and the rigidity result in Corollary~\ref{t2.3} identifies it with \(W^+\) modulo symmetries. This proves Theorem ~\ref{thm:radial-threshold} and, in particular, gives finite-time blow-up of \(W^+\) in the negative time direction.

The main new ingredient in the radial argument is a localized virial estimate adapted to data in \(\dot H^1\), where no finite-mass assumption is made.
Following the spirit of Merle and Rapha\"el~\cite{MR08}, we control the exterior nonlinear contribution by a scale-invariant Hardy-type quantity and combine this with a flux estimate. This allows us to exclude global radial solutions with a uniform negative virial gap. At the threshold, the same control also yields the no-return property needed to pass from sequential convergence to trapping near the ground-state manifold.

The rigidity statement used above concerns threshold solutions that remain near \(\mathcal W\). Modulation and virial estimates control the scaling and translation parameters and yield exponential convergence to the ground-state manifold. The classification of exponentially convergent threshold solutions then gives \(W^+\) modulo symmetries.

For non-radial data with finite variance, a Banica-type inequality, the virial identity, and modulation estimates rule out repeated departures from the ground-state manifold. Together with the preceding rigidity statement, this yields the complete classification in Theorem~\ref{thm:general-threshold}(i).

Finally, we consider the non-radial results in dimensions \(d\ge5\) under the kinetic-energy bound involving \(\kappa_d\). The role of \(\kappa_d\) is to ensure favorable energy properties for the profiles in a concentration--compactness decomposition. Below the threshold, this leads to a minimal almost-periodic solution, which is excluded by a localized virial argument, proving the {endpoint} escape estimate in Theorem~\ref{thm:general-subthreshold}(ii). { For solutions on the upper threshold sheet, the escape result in Theorem~\ref{thm:general-subthreshold}(ii), together with the scattering theory in the interior of the lower cone, allows us to exclude nonlinear profiles with energy strictly below $E(W)$ in the concentration compactness argument. Consequently, the forward orbit is compact modulo scaling and translation symmetries, and a further virial argument yields sequential convergence to the ground-state manifold $\mathcal W$.} If the solution were to leave every small neighborhood of \(\mathcal W\), compactness and stability would produce a forward-trapped threshold solution; the rigidity result would force it to be \(W^+\), contradicting its finite backward lifespan. This proves Theorem~\ref{thm:general-threshold}(ii).

\subsection*{Organization of the paper.} Section 2 collects the variational and modulation estimates and develops the localized virial identities used throughout the paper. In Section 3, we establish a bootstrap argument and a rigidity result for threshold solutions that remain close to the ground-state manifold. Sections 4 and 5 are devoted to radial solutions in the interior of the upper cone and on the upper threshold sheet, respectively. In Section 6, we turn to non-radial solutions. We first treat the finite-variance case, then use concentration-compactness and virial arguments to establish the {endpoint} escape result and the rigidity theorem on the upper threshold sheet. Finally, Appendix A determines precisely when the special solution \(W^+\) has finite variance.

\subsection*{Acknowledgments}
Part of this work was done during Z. Ma’s visits to Southeast University in December 2025 and September 2026. C. Miao was supported by the National Key R\&D program of China: No.2022YFA1005700 and NSFC Grants 12371095 and 12531005.  K. Yang was supported by  the Jiangsu Scientific Research Center of Applied Mathematics under grant BK20233002, NSFC grant 12671287, and NSF of Jiangsu (China) grant. J. Zheng was supported by  National key R\&D program of China: 2021YFA1002500 and  NSFC Grant 12671284.

\subsection*{Declaration of AI Use}

ChatGPT 5.6 Sol was used to assist in the proof of Lemma \ref{lem:radial-tail}. 
The authors also used ChatGPT 5.6 Sol and ChatGPT 6 Astra for language editing, generating Figure 1,
and checking proofs previously developed by the authors. 
The authors reviewed all AI-assisted revisions and take full responsibility for the mathematical content and conclusions of the paper.

\section{Preliminaries and localized virial estimates}\label{sec:modulation}

In this section, we collect the notation and local theory and establish variational and modulation estimates, along with localized virial identities and related bounds.

\subsection{Notation and variational estimates} We begin by setting up some notation. For nonnegative quantities $X$ and $Y$, we write $X\lesssim Y$ to denote the estimate $X\leq C Y$ for some $C>0$. If $X\lesssim Y\lesssim X$, we will write $X\sim Y$. Dependence on various parameters will be indicated by subscripts, e.g. $X\lesssim_u Y$ indicates $X\leq CY$ for some $C=C(u)$.

The $L^2$ and $\dot H^1$ inner products are defined by $(f,g)_{L^2}=\Re\int f\bar g$ and $(f,g)_{\dot H^1}=(\nabla f,\nabla g)_{L^2}$. 
For a spacetime slab $I\times\R^d$, we write $L_t^q L_x^r(I\times\R^d)$ for the Banach space of functions $u:I\times\R^d\to\C$ equipped with the norm
\[
\|u\|_{L_t^{q}L_x^r(I\times\R^d)}:=\bigg(\int_I \|u(t)\|_{L_x^r(\R^d)}^q\,dt\bigg)^{1/q},
\]
with the usual adjustments if $q$ or $r$ is infinity. When $q=r$, we abbreviate $L_t^qL_x^q=L_{t,x}^q$ and $\|f\|_{L_x^r(\R^d)}=\|f\|_{L_x^r}.$ For $1\leq r\leq\infty$, we use $r'$ to denote the dual exponent to $r$, i.e. the solution to $\frac{1}{r}+\frac{1}{r'}=1.$

The Fourier transform on $\mathbb{R}^d$ is given by
\begin{equation*}
\aligned \widehat{f}(\xi)= \F f(\xi)= (2\pi)^{-d/2}\int_{\mathbb{R}^d}e^{- ix\cdot \xi}f(x)\,dx.
\endaligned
\end{equation*}
We then define the fractional differentiation operator $|\nabla|^s=\F^{-1}|\xi|^s\F$, with the corresponding homogeneous Sobolev norm 
\[
\|f\|_{\dot{H}_x^s(\R^d)}:=\||\nabla|^s f\|_{L_x^2(\R^d)}.
\]

We denote the free Schr\"odinger propagator by $e^{it\Delta}$.  It is most naturally defined as the Fourier multiplier operator with symbol $e^{-it|\xi|^2}$, but also has a physical-space representation given by convolution with a complex Gaussian. This operator satisfies well-known estimates known as Strichartz estimates, which we state as follows.

First, we call a pair of exponents $(q,r)$ \emph{Schr\"odinger admissible} (in dimension $d$) if $2\leq q,r\leq\infty$ and $\tfrac{2}{q}+\tfrac{d}{r}=\tfrac{d}{2}.$ For a spacetime slab $I\times\R^d$, we define the Strichartz norm 
\[
\|u\|_{S^0(I)}:=\sup\big\{\|u\|_{L_t^{q}L_x^r(I\times\R^d)}:(q,r)\text{ is Schr\"odinger admissible}\big\}.
\]
We denote $S^0(I)$ to be the closure of all test functions under this norm and write $N^0(I)$ for the dual of
$S^0(I)$. We then have the following:

\begin{proposition}[Strichartz estimates, \cite{GinibreVelo, KeelTao, Strichartz}]\label{prop1}
Let $s\geq 0$ and suppose $u:I\times\R^d\to\C$ is a solution to $(i\partial_t+\Delta)u=F$. For any $t_0\in I$, 
\[
\||\nabla|^s u\|_{S^0(I)}\lesssim\||\nabla|^s u(t_0)\|_{L_x^2(\R^d)}+\||\nabla|^s F\|_{N^0(I)}.
\]
\end{proposition}

We next give some basic properties for the ground state $W$.
\begin{lemma}[Sharp Sobolev embedding, \cite{Au76, Ta76}] \label{GNsharpz}
For any $f\in \dot H^1(\R^d)$, 
\begin{equation}\label{GNinequality}
\|f\|_{L^{\frac{2d}{d-2}}_x(\mathbb{R}^d)} \leq C_d \|f\|_{\dot{H}^{1}(\mathbb{R}^d)},
\end{equation}
where 
\[
C_{d}:=\|W\|_{L^{\frac{2d}{d-2}}_x(\mathbb{R}^d)} \|W\|_{\dot{H}^{1}(\mathbb{R}^d)}^{-1}.
\]
Moreover, equality holds in \eqref{GNinequality} if and only if \(f(x)=z_0  W(\lambda_0 x+x_0)\) for some  $\lambda_0>0$, $z_0\in\C$ and $x_0\in\R^d$.
\end{lemma}
The ground state $W$ also satisfies the following Pohozaev identities:
\begin{equation*}
\|W\|_{\dot{H}^1(\R^d)}^2=\|W\|^{\frac{2d}{d-2}}_{L^{\frac{2d}{d-2}}(\R^d)} \quad \mbox{and }\ \quad E(W)=\tfrac{1}{d}\|W\|_{\dot{H}^1(\R^d)}^2.
\end{equation*}

We will need the following two elementary lemmas.

\begin{lemma}\label{lem:subthreshold-branch}
Let $u:I\times\R^d\to\C$ be a solution to \eqref{1.1} satisfying \eqref{hyp:subthreshold}. Then, for all $t\in I$,
\begin{equation}\label{supk}
 \|\nabla u(t)\|_{L_x^2}>\|\nabla W\|_{L_x^2},\quad
 K(u(t))\leq-2^*\bigl[E(W)-E(u)\bigr]<0.
\end{equation}
\end{lemma}

\begin{proof}
If $\|\nabla u(t_0)\|_{L_x^2}=\|\nabla W\|_{L_x^2}$ at some
time $t_0\in I$, Lemma~\ref{GNsharpz} gives
$\|u(t_0)\|_{L_x^{2^*}}^{2^*}
\leq\|\nabla W\|_{L_x^2}^2$. Then, we have
\[
 E(W)>E(u)\geq\left(\frac12-\frac{d-2}{2d}\right)
              \|\nabla W\|_{L_x^2}^2=E(W).
\]
This contradiction confirms the first assertion in \eqref{supk}. The expression of $E(u)$ and \eqref{def:deltaK} imply
\begin{equation}\label{Ku}
     K(u(t))=2^*E(u)
                -\frac2{d-2}\|\nabla u(t)\|_{L_x^2}^2.
\end{equation}
$ \|\nabla u(t)\|_{L_x^2}>\|\nabla W\|_{L_x^2}$ and
$\|\nabla W\|_{L_x^2}^2=dE(W)$ complete the proof.
\end{proof}

\begin{lemma}\label{lem:threshold-branch}
Suppose $u$ is a solution to \eqref{1.1} satisfying \eqref{hyp:bounded}. Then
\begin{equation}\label{eq:positive-delta}
 \delta(u(t))=\|\nabla u(t)\|_{L_x^2}^2-\|\nabla W\|_{L_x^2}^2>0
\end{equation}
throughout its lifespan $I$, and
\begin{equation}\label{eq:Kdelta}
 K(u(t))=-\frac2{d-2}\delta(u(t)).
\end{equation}
In particular,
\begin{equation}\label{potg}
 \|u(t)\|_{L_x^{2^*}}^{2^*}
 >\|W\|_{L_x^{2^*}}^{2^*}.
\end{equation}
If $u$ is global forward in time, then $S_{[0,\infty)}(u)=\infty$.
\end{lemma}

\begin{proof}
If $\delta(u(t_0))=0$ for some $t_0\in I$, then energy conservation gives
\[
 \|\nabla u(t_0)\|_{L_x^2}^2
 =\|u(t_0)\|_{L_x^{2^*}}^{2^*}
 =\|\nabla W\|_{L_x^2}^2=\|W\|_{L_x^{2^*}}^{2^*}.
\]
By Lemma~\ref{GNsharpz}, $u(t_0)$ equals $W$ up to symmetries, and so does $u(t)$ for all $t\in I$. This contradicts {\eqref{hyp:bounded}}.
Identity \eqref{eq:Kdelta} follows directly from
$E(u)=E(W)=\|\nabla W\|_{L_x^2}^2/d$.  
$E(u)=E(W)$ and \eqref{eq:positive-delta} imply
\eqref{potg}.
If $u$ scatters forward in time, then along a sequence $t_n\to\infty$,
\[
 \|u(t_n)\|_{L_x^{2^*}}\to0.
\]
This contradicts \eqref{potg} and proves the last assertion.
\end{proof}

\subsection{Local theory} We turn to local well-posedness and stability for \eqref{1.1}. An important role is played by the $L_{t,x}^{\frac{2(d+2)}{d-2}}$ spacetime norm of solutions. The associated \emph{scattering size} is
\[
S_I(u) = \|u\|_{L_{t,x}^{\frac{2(d+2)}{d-2}}(I\times\R^d)}^{\frac{2(d+2)}{d-2}}.
\]  
In particular, solutions may be extended in time as long as this scattering size remains finite, and a global solution with finite $L_{t,x}^{\frac{2(d+2)}{d-2}}$-norm necessarily scatters in $\dot H^1$ in both directions. In light of this, we define \emph{blow-up} to mean infinite $L_{t,x}^{\frac{2(d+2)}{d-2}}$-norm.

\begin{definition}[Blow-up]\label{s2.7def}
Let $u:I\times\R^d\to\C$ be a maximal-lifespan solution to \eqref{1.1}. We say $u$ \emph{blows up forward in time} if there exists a time $t_1\in I$ such that $S_{[t_1, \sup I)}(u)=\infty$, and that $u$ \emph{blows up backward in time} if there exists a time $t_1 \in I$ such that $S_{(\inf I, t_1]}(u) = \infty.$
\end{definition}

We then have the following local well-posedness result for \eqref{1.1}.

\begin{theorem}[Local well-posedness, \cite{KVnote, VisanOberwolfach, Tao1}]\label{T:local}
Let $u_0\in\dot H_x^1(\R^d)$ and $t_0\in \R$. There exists a unique maximal-lifespan solution $u:I\times\R^d \to \C$ to
\eqref{1.1} with initial data $u|_{t=t_0}=u_0$.  Furthermore, we have the following:
\begin{enumerate}
\item[(i)] $I$ is an open neighborhood of $t_0$.
\item[(ii)] If $\sup(I)$ is finite, then $u$ blows up forward in time; if $\inf(I)$ is finite,
then $u$ blows up backward in time.
\item[(iii)] If $\sup(I)=\infty$ and $u$ does not blow up forward in time, then $u$ scatters forward in time.  The analogous statement holds backward in time.
\item[(iv)] If $\|\nabla u_0\|_{L^2}$ is sufficiently small, then $u$ is global and scatters in both time directions, with $S_\R(u)\lesssim \|\nabla u_0\|_{L^2}^{\frac{2(d+2)}{d-2}}$.
\end{enumerate}
\end{theorem}

In addition, we will need the following stability result for \eqref{1.1}.

\begin{lemma}[Stability, \cite{KVnote, VisanOberwolfach, Tao1}]\label{pertu}  For any $E,L > 0$ and $\eps > 0$ there exists $\delta > 0$
such that the following holds: Suppose $\tilde u: I \times \R^d \to \C$ satisfies
\begin{equation*}
\bigl\| \nabla \bigl[(i\partial_t+\Delta)\tilde u +|\tilde u|^{\frac4{d-2}} \tilde u \bigr] \bigr\|_{L_{t,x}^{\frac{2(d+2)}{d+4}}(I\times\R^d)} \leq \delta
\end{equation*}
as well as
\begin{align*}
\|\tilde u\|_{L_t^\infty \dot H^1_x(I\times\R^d)}\leq E \quad \text{and} \quad S_I(\tilde u)\leq L.
\end{align*}
If $t_0 \in I$ and $u_0 \in \dot H^1_x(\R^d)$ are such that
\[
\|\tilde u(t_0)-u_0\|_{\dot H^1_x(\R^d)} \leq \delta,
\]
then there exists a solution $u: I \times \R^d \to \C$ to \eqref{1.1} with $u|_{t=t_0}= u_0$ such that
\begin{align*}
\|\tilde u -u\|_{L_t^\infty \dot H^1_x(I\times\R^d)} + S_I(\tilde u - u) \leq \eps.
\end{align*}
\end{lemma}

We also write
\[
 \|u\|_{\dot S^1(I)}:=\|\nabla u\|_{S^0(I)},\qquad
 \|F\|_{\dot N^1(I)}:=\|\nabla F\|_{L_{t,x}^{\frac{2(d+2)}{d+4}}(I\times\R^d)}.
\]

\subsection{Modulation estimates}

We next collect some results related to modulation analysis for \eqref{1.1}.

For $\theta\in\R$, $\lambda>0$, and $x_0\in\R^d$, set
\[
 f_{[\theta,\lambda,x_0]}(x)
 :=e^{i\theta}\lambda^{-\frac{d-2}{2}}
 f\left(\frac{x-x_0}{\lambda}\right).
\]
Thus, if $f_{[\theta,\lambda,x_0]}=W+h$, then
\[
 f(x)=e^{-i\theta}\lambda^{\frac{d-2}{2}}(W+h)(\lambda x+x_0).
\]

\begin{lemma}[Modulation parameters, \protect\cite{DM08,MMMZ25}]\label{Mod1}
There exists $\delta_0>0$ such that, for every $f\in\dot H^1(\mathbb{R}^d)$ with
$E(f)=E(W)$ and $|\delta(f)|<\delta_0$, there is a unique triple
$(\theta,\lambda,\tilde x)\in
\mathbb{R}/2\pi\mathbb{Z}\times(0,\infty)\times\R^d$ such that
$f_{[\theta,\lambda,\tilde x]}$ lies in a fixed sufficiently small
$\dot H^1$ neighborhood of $W$ and
\[
f_{[\theta,\lambda,\tilde{x}]}\perp iW,W_1,\partial_jW,
\qquad j\in\{1,\ldots,d\},
\]
with respect to the real $\dot H^1$ inner product
$(f,g)_{\dot H^1}:=\Re\int_{\R^d}\nabla f\cdot\overline{\nabla g}\,dx$,
where
\[
W_1:=\frac{d-2}{2}W+x\cdot\nabla W.
\]
Moreover, the map $f\mapsto(\theta,\lambda,\tilde{x})$ is $C^1$.
\end{lemma}

We suppose $u:I\times\R^d\to\C$ is a solution to \eqref{1.1} satisfying
\begin{equation}\label{u-supercritical}
E(u)=E(W), \|\nabla u(t)\|_{L_x^2}>\|\nabla W\|_{L_x^2}\qtq{and} \delta(u(t))<\delta_0 \qtq{on}I.
\end{equation}
We further define
\begin{equation}\label{defAorth}
\mathcal{A}:=\{W, iW, W_1, \partial_{1} W, \cdots, \partial_{d} W\}.
\end{equation}
We choose $(\theta(t),\lambda(t),x(t))$ corresponding to $u(t)$ as in Lemma~\ref{Mod1} and further define $\alpha(t)$ via
\[
1+\alpha(t)=\frac{1}{\|W\|_{\dot H^1}^2}(u_{[\theta(t),\lambda(t),x(t)]},W)_{\dot{H}^1}.
\]
We can then write
\begin{equation}\label{ortho}
u_{[\theta(t),\lambda(t),x(t)]}(t)=\bigl[1+\alpha(t)\bigr]W+\tilde{u}(t),\qtq{with} \tilde u(t)\in \mathcal{A}^\perp.
\end{equation}
Finally, we define $v(t)$ by
\[
v(t):=\alpha(t)W+\tilde{u}(t)=u_{[\theta(t),\lambda(t),x(t)]}(t)-W.
\]

We will need various estimates on the modulation parameters, which we record in the following lemma.

\begin{lemma}[Modulation estimates, \protect\cite{DM08,MMMZ25}]\label{modula}
Suppose $u:I\times\R^d\to\C$ is a solution to \eqref{1.1} satisfying \eqref{u-supercritical}. If $\delta_0$ is sufficiently small, then the following estimates hold on $I$:
\begin{align}
|\alpha(t)|\sim \|v(t)\|_{\dot{H}^1(\R^d)}\sim \|\tilde{u}(t)\|_{\dot{H}^1(\R^d)} \sim \delta(u(t)), \label{equivq1} \\
\bigl|{x^{\prime}(t)}-\tfrac{\lambda^{\prime}(t)}{\lambda(t)}x(t)\bigr|+|\alpha^{\prime}(t)|+|\theta^{\prime}(t)|+\bigl|\tfrac{\lambda^{\prime}(t)}{\lambda(t)}\bigr|\lesssim\lambda(t)^2 \delta(u(t)).\label{Mod2}
\end{align}
\end{lemma}


\subsection{Localized virial identities and estimates}

Fix a nonnegative radial function $\varphi\in C_c^\infty(\R^d)$
and a sufficiently large constant $C_0>2$ 
such that
\begin{equation}\label{start1}
 \begin{gathered}
 \varphi(x)=|x|^2\quad(|x|\leq1),\qquad
 \varphi(x)=0\quad(|x|\geq C_0),\qquad
  \varphi''(r)\leq2,\\
 0\leq\varphi(x)\leq|x|^2,\qquad
 |\partial^\alpha\varphi(x)|\lesssim_\alpha|x|^{2-|\alpha|},
 \end{gathered}
\end{equation}
where the radial derivative is defined by $\varphi'(r)$ with $r=|x|$.

For $R>0$, set $\varphi_R(x):=R^2\varphi(x/R)$ and define
\begin{equation}\label{eq:localized-virial-definitions}
 V_R(f):=\int_{\R^d}\varphi_R|f|^2\,dx,\qquad
 M_R(f):=2\Im\int_{\R^d}\nabla\varphi_R\cdot\nabla f\,\bar f\,dx.
\end{equation}

\begin{lemma}[Localized virial identities, \cite{DM08}]\label{lem:localized-virial}
Let $u$ be a solution to \eqref{1.1}. Then
\begin{align}
 \frac{d}{dt}V_R(u(t))&=M_R(u(t)),\label{eq:localized-variance}\\
 \frac{d}{dt}M_R(u(t))&=8K(u(t))+F_R(u(t)),\label{eq:localized-virial}
\end{align}
where
\begin{align}
 F_R(f)
 &:=-\int_{\R^d}\Delta^2\varphi_R|f|^2\,dx
 -\frac4d\int_{\R^d}(\Delta\varphi_R-2d)
 |f|^{2^*}\,dx\notag\\
 &\quad+4\Re\int_{\R^d}
 (\partial_{jk}\varphi_R-2\delta_{jk})
 \partial_jf\,\overline{\partial_kf}\,dx,
 \label{eq:virial-remainder}
\end{align}
and repeated indices are summed. If $E(u)=E(W)$, then
\begin{equation}\label{thlocalv}
 \frac{d}{dt}M_R(u(t))
 =-\frac{16}{d-2}\delta(u(t))+F_R(u(t)).
\end{equation}
For every fixed element $\widetilde W$ of the ground-state manifold,
\begin{equation}\label{eq:ground-state-virial-zero}
 M_R(\widetilde W)=F_R(\widetilde W)=0.
\end{equation}
\end{lemma}

\begin{proof}
Direct computations involving \eqref{1.1} yield \eqref{eq:localized-variance} and \eqref{eq:localized-virial}.

Clearly, \eqref{thlocalv} follows from  \eqref{eq:localized-virial} and \eqref{eq:Kdelta}. A fixed ground state is stationary and real-valued
up to symmetries, showing \eqref{eq:ground-state-virial-zero}.
\end{proof}
We collect several elementary estimates that will be used together with these identities. In the following, no radial symmetry is assumed and we introduce the  \emph{exterior Hardy mass} 
\begin{equation}\label{eq:radial-Hardy-tail}
 H_R(f):=\int_{\R^d}\chi\left(\frac{|x|}{R}\right)
 \frac{|f(x)|^2}{|x|^2}\,dx,
\end{equation}
where  $\chi\in C^\infty([0,\infty))$ satisfies
\[
 0\leq\chi\leq1,\qquad
 \chi(r)=0\ (r\leq\tfrac12),\qquad
 \chi(r)=1\ (r\geq1),\qquad
 |\chi'(r)|^2\lesssim\chi(r).
\]
The last condition is obtained by squaring a smooth cutoff with the
first three properties. 
Hardy's inequality shows that $V_R(f)$ and $H_R(f)$ are finite for
$f\in\dot H^1(\mathbb{R}^d)$.

\begin{lemma}[Basic virial estimates]\label{lem:virial-estimates}
For $f,g\in\dot H^1(\R^d)$ and $R>0$,
\begin{align}
 |M_R(f)|&\lesssim R^2\|\nabla f\|_{L_x^2}^2,
 \label{eq:virial-M-bound}\\
 |M_R(f)-M_R(g)|
 &\lesssim R^2(\|\nabla f\|_{L_x^2}+\|\nabla g\|_{L_x^2})
               \|f-g\|_{\dot H^1},
 \label{eq:virial-action-difference}\\
 |F_R(f)|&\lesssim\int_{|x|\geq R}
  |\nabla f|^2+|f|^{2^*}
             +\frac{|f|^2}{|x|^2} \,dx,
 \label{eq:virial-tail}\\
 F_R(f)&\lesssim H_R(f)+\int_{|x|>R}|f|^{2^*}\,dx.
 \label{oneside}
\end{align}
We also have
\begin{equation}\label{eq:virial-tail-L2}
 \begin{aligned}
 |F_R(f)|
 \lesssim\int_{|x|\geq R}
          (|\nabla f|^2+|f|^{2^*})\,dx
 +\bigl(\int_{R\leq|x|\leq C_0 R}
          |f|^{2^*}\,dx\bigr)^{\frac{d-2}{d}}.
 \end{aligned}
\end{equation}
For a ground-state $\widetilde W$ in the ground-state manifold and $h\in\dot H^1(\mathbb{R}^d)$,
\begin{equation}\label{VED}
 \begin{aligned}
 |F_R(\widetilde W+h)|
 &\lesssim\left(\|\nabla\widetilde W\|_{L_x^2(|x|>R)}
       +\|\widetilde W/|x|\|_{L_x^2(|x|>R)}\right)\|h\|_{\dot H^1}
       +\|h\|_{\dot H^1}^2\\
 &\quad+\|\widetilde W\|_{L_x^{2^*}(|x|>R)}^{\frac{d+2}{d-2}}
         \|h\|_{\dot H^1}
       +\|h\|_{\dot H^1}^{2^*}.
 \end{aligned}
\end{equation}
In particular, if $E(f)=E(W)$ and $\delta(f)>0$, then
\begin{equation}\label{rac0}
 |M_R(f)|\lesssim R^2\delta(f).
\end{equation}
Finally, for each fixed $f\in\dot H^1(\mathbb{R}^d)$,
\begin{equation}\label{eq:virial-initial-smallness}
 H_R(f)\to0,\qquad
 V_R(f)=o(R^4),\qquad M_R(f)=o(R^2)
 \quad(R\to\infty).
\end{equation}
The constants depend only on $d$ and the fixed cutoff functions.
\end{lemma}

\begin{proof}
The definition of $\varphi_R$ together with
H\"older's inequality and Sobolev embedding yield
\begin{align*}
 |M_R(f)|
 &\lesssim R\|\nabla f\|_{L_x^2}
                 \|f\|_{L_x^2(|x|\leq C_0 R)}
 \lesssim R^2\|\nabla f\|_{L_x^2}^2,\\
 |M_R(f)-M_R(g)|
 &\lesssim R\int_{|x|\leq C_0 R}
      (|\nabla f|\,|f-g|+|\nabla(f-g)|\,|g|)\,dx\\
 &\lesssim R^2(\|\nabla f\|_{L_x^2}+\|\nabla g\|_{L_x^2})
                 \|f-g\|_{\dot H^1}.
\end{align*}

All coefficients in \eqref{eq:virial-remainder} vanish on $|x|<R$.
Moreover, $\Delta^2\varphi_R$ is supported in
$R\leq|x|\leq C_0 R$ and satisfies
$|\Delta^2\varphi_R|\lesssim R^{-2}\lesssim|x|^{-2}$ there.
This proves \eqref{eq:virial-tail}. By the definition of $\varphi$, in particular the bound $\varphi''(r)\le2$, $\nabla^2\varphi_R-2I$ is negative semi-definite, its contribution in $F_R(f)$ is
non-positive. The same bounds and $\chi(|x|/R)=1$ when $|x|\geq R$
then prove \eqref{oneside}.
For \eqref{eq:virial-tail-L2}, estimate the truncated mass term instead by
\[
 \left|\int_{\R^d}\Delta^2\varphi_R|f|^2\,dx\right|\lesssim  \int_{R\leq|x|\leq C_0 R}\frac{|f|^2}{R^2}\,dx
 \lesssim
 \bigl(\int_{R\leq|x|\leq C_0 R}
 |f|^{2^*}\,dx\bigr)^{\frac{d-2}{d}}.
\]

Note that $F_R(\widetilde W)=0$, \eqref{VED} follows from the triangle inequality and
\begin{align*}
 \int_{|x|>R}(2|\nabla\widetilde W|\,|\nabla h|+|\nabla h|^2)\,dx
 &\lesssim\|\nabla\widetilde W\|_{L_x^2(|x|>R)}\|h\|_{\dot H^1}
          +\|h\|_{\dot H^1}^2;
\end{align*}
\begin{align*}
 \left|\int\Delta^2\varphi_R
       (|\widetilde W+h|^2-|\widetilde W|^2)\,dx\right|
 &\lesssim\int_{|x|>R}
       \frac{2|\widetilde W|\,|h|+|h|^2}{|x|^2}\,dx\\
 &\lesssim\|\widetilde W/|x|\|_{L_x^2(|x|>R)}\|h\|_{\dot H^1}
           +\|h\|_{\dot H^1}^2;
\end{align*}
\begin{align*}
 \int_{|x|>R}
 \bigl||\widetilde W+h|^{2^*}
              -|\widetilde W|^{2^*}\bigr|\,dx  &\lesssim\int_{|x|>R}
   (|\widetilde W|^{\frac{d+2}{d-2}}|h|
           +|h|^{2^*})\,dx\\
 & \lesssim
 \|\widetilde W\|_{L_x^{2^*}(|x|>R)}^{\frac{d+2}{d-2}}
 \|h\|_{\dot H^1}+\|h\|_{\dot H^1}^{2^*}.
\end{align*}

Suppose $E(f)=E(W)$ and $0<\delta(f)<\delta_0$.
Lemma~\ref{modula}, applied at the initial time of the solution
with data $f$, gives a ground state $\widetilde W$ such that
$\|f-\widetilde W\|_{\dot H^1}\lesssim\delta(f)$.
Using \eqref{eq:virial-action-difference},
$M_R(\widetilde W)=0$, and
$\|\nabla f\|_{L_x^2}^2=\|\nabla W\|_{L_x^2}^2+\delta(f)$,
we obtain \eqref{rac0}. If $\delta(f)\geq\delta_0$,
then \eqref{eq:virial-M-bound} gives
\[
 |M_R(f)|\lesssim R^2(\|\nabla W\|_{L_x^2}^2+\delta(f))
 \lesssim R^2\delta(f).
\]

For a fixed $f\in\dot H^1(\mathbb{R}^d)$, as $R\to\infty$, Hardy's inequality and dominated
convergence imply $H_R(f)\to0$, and {a density argument} yields
\[
 R^{-2}\int_{|x|\leq C_0 R}|f(x)|^2\,dx\to0.
\]
Consequently,
\begin{align*}
 V_R(f)&\lesssim R^2\int_{|x|\leq C_0 R}|f|^2\,dx=o(R^4),\\
 |M_R(f)|&\lesssim R\|\nabla f\|_{L_x^2}
                  \|f\|_{L_x^2(|x|\leq C_0 R)}=o(R^2).
\end{align*}
This proves \eqref{eq:virial-initial-smallness}.
\end{proof}

\begin{lemma}[Exterior Hardy mass estimates]\label{lem:Hardy-flux}
Let $u:J\times\R^d\to\C$ be a solution to \eqref{1.1}. For $R>0$
and $t_1<t_2$ in $J$,
\begin{equation}\label{rHroot}
 H_R(u(t_2))^{1/2}
 \leq H_R(u(t_1))^{1/2}
       +\frac C{R^2}\int_{t_1}^{t_2}\|\nabla u(t)\|_{L_x^2}\,dt.
\end{equation}
If $E(u)=E(W)$ and $\delta(u(t))>0$, then
\begin{equation}\label{refHardy}
 \left|\frac d{dt}H_R(u(t))\right|
 \lesssim R^{-2}\delta(u(t)).
\end{equation}
\end{lemma}

\begin{proof}
{A computation analogous to the derivation of \eqref{eq:localized-variance} gives}
\[
 \frac d{dt}H_R(u(t))
 =2\Im\int_{\R^d}
 \nabla\bigl(\frac{\chi(|x|/R)}{|x|^2}\bigr)
 \cdot\nabla u\,\bar u\,dx.
\]
For $r=|x|>0$, direct differentiation yields
\[
 \nabla\bigl(\frac{\chi(r/R)}{r^2}\bigr)
 =\bigl(\frac{\chi'(r/R)}{Rr^2}
              -\frac{2\chi(r/R)}{r^3}\bigr)\frac{x}{r}.
\]
Both terms vanish for $r\leq R/2$. 
Recall  $ |\chi'(r)|^2\lesssim\chi(r)\le 1$, we have
\begin{align*}
 \left|\nabla\bigl(\frac{\chi(r/R)}{r^2}\bigr)\right|^2
 &\lesssim\frac{|\chi'(r/R)|^2}{R^2r^4}
                 +\frac{\chi(r/R)^2}{r^6}
 \lesssim R^{-4}\frac{\chi(r/R)}{r^2}.
\end{align*}
Cauchy--Schwarz then implies
\begin{equation}\label{Hrdbd}
 \left|\frac d{dt}H_R(u(t))\right|
 \lesssim R^{-2}H_R(u(t))^{1/2}\|\nabla u(t)\|_{L_x^2}.
\end{equation}
Then \eqref{rHroot} follows from integrating this inequality.

For \eqref{refHardy}, if
$0<\delta(u(t))<\delta_0$ and denote 
$$
\widetilde{W}(t):= \lambda(t)^{\frac{d-2}{2}} e^{-i\theta(t)} W(\lambda(t) x + x(t)).
$$ 
Lemma~\ref{modula} and Hardy's inequality yield 
\begin{align*}
 \left|\frac d{dt}H_R(u(t))\right|
 &\lesssim R^{-2}\| u(t)\|_{\dot H^1}
              \|\tfrac{u(t)-\widetilde W}{|x|}\|_{L_x^2}  +R^{-2}\| u(t)-\widetilde W\|_{\dot H^1}
              \|\tfrac{\widetilde W}{|x|}\|_{L_x^2}\\
 &\lesssim R^{-2}\delta(u(t)).
\end{align*}
If $\delta(u(t))\geq\delta_0$, \eqref{Hrdbd} then implies
\[
 \left|\frac d{dt}H_R(u(t))\right|
 \lesssim R^{-2}\|\nabla u(t)\|_{L_x^2}^2
 \lesssim R^{-2}\delta(u(t)).
\]
This completes the proof.
\end{proof}

\section{A bootstrap argument}\label{sec:local-rigidity}

In this section, we prove a rigidity result, which asserts that a supercritical threshold solution that is global forward in time and stays sufficiently close to the ground-state orbit must coincide with $W^+$ modulo symmetries.

Throughout this section, we suppose that $u:I\times\R^d\to\mathbb{C}$ is a solution to \eqref{1.1} satisfying $E(u)=E(W)$ and $\|\nabla u(0)\|_{L_x^2}>\|\nabla W\|_{L_x^2}$.

For any interval $\tilde{I}\subset I$ satisfying
\[
\sup_{t\in \tilde{I}} \delta(u(t)) \leq \eta_*,
\]
for some $\eta_{\ast}\ll1$, by Lemma~\ref{modula}, there exist continuous functions
\begin{equation*}
\theta : \tilde{I}\rightarrow \R,\quad x:  \tilde{I}\rightarrow \mathbb{R}^d ,\quad \lambda: \tilde{I}\rightarrow \mathbb{R}_+, \quad \alpha: \tilde{I}\rightarrow \mathbb{R}
\end{equation*}
such that if
\begin{equation}\label{decompz}
v(t):=\alpha(t)W+\tilde{u}(t)=u_{[\theta(t),\lambda(t),x(t)]}(t)-W,
\end{equation}
then
\begin{align}
&|\alpha(t)|\sim \|v(t)\|_{\dot{H}^1}\sim \|\tilde{u}(t)\|_{\dot{H}^1}\sim \delta(u(t)),\label{mo1}\\
&\bigl|{x^{\prime}(t)}-\tfrac{\lambda^{\prime}(t)}{\lambda(t)}x(t)\bigr|+|\alpha^{\prime}(t)|+|\theta^{\prime}(t)|+\bigl|\tfrac{\lambda^{\prime}(t)}{\lambda(t)}\bigr|\leq C\lambda(t)^2 \delta(u(t)). \label{mo2}
\end{align}

After a time translation and a fixed spatial translation, we may assume
\[
 \inf{\tilde{I}}=0\qtq{and}x(0)=0.
\]
We now run a bootstrap argument based on a localized virial estimate (incorporating the modulation analysis) to prove that the scale  and translation parameters are nearly constant. {Unlike ~\cite{MMMZ25}, the following theorem does not require an a priori bound on the scaling parameter.}
\begin{theorem}[Bootstrap regime]\label{Ret2.3}
The following holds provided $\eta_*$ is sufficiently small.  On any interval $[0,T]\subset \tilde{I}$ such that
\[
\max_{t\in[0, T]}|{x(t)}|\leq 1,
\]
we have
\begin{equation}
\max_{t\in[0, T]}\lambda(t)\leq\tfrac{3}{2}\min_{t\in[0, T]}\lambda(t),\quad  \max\limits_{t\in[0, T]}|{x(t)}|\leq \tfrac{1}{2}.
\end{equation}

Consequently,
\begin{equation}\label{bootse1}
\sup_{t\in\tilde I}\lambda(t)\leq\tfrac32\inf_{t\in\tilde I}\lambda(t),\quad
\sup_{t\in\tilde I}|x(t)|\leq\tfrac12.
\end{equation}
\end{theorem}

\begin{proof}
Set
\[
 \lambda_{\min}:=\min_{t\in[0,T]}\lambda(t),\qquad
 \lambda_{\max}:=\max_{t\in[0,T]}\lambda(t).
\]
Let $M_R$ and $F_R$ be as in Lemma~\ref{lem:localized-virial}, and write
\[
 \widetilde W(t,x):=e^{-i\theta(t)}\lambda(t)^{\frac{d-2}{2}}
 W(\lambda(t)x+x(t)),\qquad h(t):=u(t)-\widetilde W(t).
\]
By \eqref{mo1}, $\|h(t)\|_{\dot H^1}\lesssim\delta(u(t))$.
We show that for a sufficiently large fixed $C$ and
$R=C\lambda_{\min}^{-1}$,
\begin{equation}\label{step2}
 \frac{d}{dt}M_R(u(t))\leq-\frac8{d-2}\delta(u(t)).
\end{equation}
Fix $0<\eta<1$. Since $|x(t)|\leq1$ and $R\lambda(t)\geq C$,
the change of variables $y=\lambda(t)x+x(t)$ gives, for $C$
sufficiently large,
\begin{align*}
 & \int_{|x|\geq R} 
 |\nabla\widetilde W(t,x)|^2
 +|\widetilde W(t,x)|^{2^*}
 +\frac{|\widetilde W(t,x)|^2}{|x|^2} \,dx \\
 = &\int_{|y-x(t)|\geq R\lambda(t)} 
 | \nabla W(y)|^2
 +| W(y)|^{2^*}
 +\frac{| W(y)|^2}{|y-x(t)|^2} \,dy\leq\eta^2.
\end{align*}
Indeed, from $R=C\lambda_{\min}^{-1}$,  $|y-x(t)|\geq R\lambda(t)\ge  C$, and consequently
$|y|\geq C-1$ as $|x(t)|\le 1$. 
By \eqref{VED}, we obtain
\begin{align*}
 |F_R(u(t))|
 &\lesssim \eta\|h(t)\|_{\dot H^1}
 +\|h(t)\|_{\dot H^1}^2
 +\eta^{\frac{d+2}{d}}\|h(t)\|_{\dot H^1}
 +\|h(t)\|_{\dot H^1}^{2^*}\\
 &\lesssim(\eta+\eta_*)\delta(u(t)).
\end{align*}
Choose $\eta$ sufficiently small, fix the corresponding $C$, and then
choose $\eta_*$ sufficiently small. 
Then \eqref{step2} follows from \eqref{thlocalv} directly.

Integrating \eqref{step2} and using \eqref{rac0}, we find
\[
 \int_0^T\delta(u(t))\,dt
 \lesssim\lambda_{\min}^{-2}
 \bigl[\delta(u(0))+\delta(u(T))\bigr]
 \lesssim  \frac{\eta_*}{\lambda_{\min}^2}.
\]
The implicit constants may depend on the fixed choice of $C$ {but are independent of $\eta_*$}. Together with \eqref{mo2}, this yields
\[
 \frac{1}{\lambda^2_{\min}}-\frac{1}{\lambda^2_{\max}}\le \int_0^T\left|\frac{\lambda'(t)}{\lambda(t)^3}\right|\,dt \lesssim \int_0^T \delta(u(t))\,dt
\lesssim  \frac{\eta_*}{\lambda_{\min}^2}.
\]
Taking smaller $\eta_*$ if necessary, we obtain
$\lambda_{\max}\leq\frac32\lambda_{\min}$. Similarly,
\begin{align*}
 \int_0^T|x'(t)|\,dt
 &\leq\int_0^T\left|x'(t)-\frac{\lambda'(t)}{\lambda(t)}x(t)\right|\,dt
 +\int_0^T\left|\frac{\lambda'(t)}{\lambda(t)}\right||x(t)|\,dt\\
& \lesssim \int_0^T \lambda_{\max}^2\delta(u(t))\,dt \lesssim\eta_*\lambda_{\min}^{-2}\lambda_{\max}^2
 \lesssim\eta_*.
\end{align*}
Since $x(0)=0$, taking smaller $\eta_*$ if necessary gives
$\sup_{t\in[0,T]}|x(t)|\leq\frac12$. A continuity argument then completes the proof.
\end{proof}

As a corollary, a solution on {the upper threshold sheet} that is global forward in time and remains sufficiently close to the ground-state manifold must coincide with $W^+$ modulo symmetries. We postpone its proof to the end of this section.
\begin{corollary}\label{t2.3} Suppose that $u$ is a solution to \eqref{1.1} satisfying $E(u)=E(W)$ and $\|\nabla u(0)\|_{L_x^2}>\|\nabla W\|_{L_x^2}$.
Suppose that $u$ is global forward in time.  Then there exists $\eta_*>0$ sufficiently small such that the following holds:  If
\begin{equation}\label{ppl}
    \sup_{t\in[0,\infty)} \delta(u(t)) \leq \eta_*,
\end{equation}
then there exist $\lambda_0>0$, $\theta_0\in\R$, $x_0\in\R^d$, and $T\in\R$ such that
\[
u(t,x)=e^{-i\theta_0}\lambda_0^{\frac{d-2}{2}}W^{+}(\lambda_0^2(t-T),\lambda_0x+x_0).
\]
\end{corollary}

We now use Theorem~\ref{Ret2.3} to prove the exponential convergence of the modulation parameters for threshold solutions that stay sufficiently close to the ground-state manifold.

\begin{proposition}[Exponential convergence]\label{exp} Suppose that $u$ satisfies the assumptions of Corollary~\ref{t2.3}, including \eqref{ppl}. Then there exist $\lambda_0>0$, $x_0\in\R^d$, $\theta_0\in\R$, and ${c}_0>0$ such that
\begin{equation}\label{exp-convergence}
\begin{aligned}
&|x(t)-x_0|+|\tfrac{\lambda(t)}{\lambda_0}-1|+|\theta(t)-\theta_0|\lesssim e^{-{c}_0 t},\\
& |\alpha(t)|\sim\delta(u(t))\lesssim e^{-{c}_0 t}.
\end{aligned}
\end{equation}

\end{proposition}

\begin{proof}

Applying Theorem~\ref{Ret2.3} with $\tilde{I}=[0,\infty)$ and recalling that we have normalized to  $x(0)=0$, we have
\[
\lambda(t)\sim \lambda(0) \qtq{and} |x(t)| \leq 1\qtq{for all}t\geq 0.
\]
We now repeat the modulated virial estimate used to prove Theorem~\ref{Ret2.3} to obtain the estimate
\[
\int_{t_1}^{t_2} \delta(u(t))\,dt \lesssim \delta(u(t_1))+\delta(u(t_2))
\]
for any $t_2>t_1\geq 0$. In particular, $t\mapsto \delta(u(t))\in L_t^1((0,\infty))$, so that there exists $T_n\to\infty$ such that $\delta(u(T_n))\to 0$.  Applying the estimate above on intervals $[t,T_n]$ thus yields
\[
\int_t^\infty \delta(u(s))\,ds \lesssim \delta(u(t))
\qquad(t\geq0).
\]
By Gronwall's inequality, it follows that
\[
\int_t^\infty \delta(u(s))\,ds \lesssim e^{-c_0t}.
\]
By \eqref{mo2} and the bounds on $\lambda$, $|\alpha'(t)|\lesssim\delta(u(t))$.
Since $\alpha(T_n)\to0$, integration from $t$ to $T_n$ and passage to
the limit give
\[
 \delta(u(t))\lesssim|\alpha(t)|
 \lesssim\int_t^\infty\delta(u(s))\,ds\lesssim e^{-c_0t}.
\]
The remaining estimates follow by integrating \eqref{mo2}, using
$|x(t)|\leq1$ and $\lambda(t)\sim\lambda(0)$. In particular,
$\lambda$ has a positive limit, and a continuous lift of $\theta$
converges as $t\to\infty$.
\end{proof}

Completing the proof of Corollary~\ref{t2.3} requires one final ingredient:
\begin{proposition}\label{2WW}
Let $u:I\times\R^d\to\mathbb{C}$ be a solution to \eqref{1.1} satisfying
\[
E(u)=E(W),\quad\|u_0\|_{\dot{H}^1(\R^d)}>\|W\|_{\dot{H}^1(\R^d)}\qtq{and}\sup I=\infty.
\]
Suppose further that
\begin{equation}\label{decayf}
\left\|u(t)-W\right\|_{\dot{H}_x^1} \leq Ce^{-ct},\quad \forall \ t\geq 0
\end{equation}
for some $c,C>0$. Then there exists $T\in \mathbb{R}$ such that $u(t)=W^{+}(t-T).$
\end{proposition}
The radial result is proved in \cite{CFR22,DM08,LZ09}; the non-radial
extension is described in \cite{MMMZ25}.{\footnote{{In \cite{MMMZ25}, the non-radial extension is formulated for the
lower-kinetic-energy branch; the same argument applies to the upper-kinetic-energy branch.}}} With this proposition, we are ready to prove Corollary~\ref{t2.3}.

\begin{proof}[Proof of Corollary~\ref{t2.3}]
Proposition~\ref{exp}, \eqref{mo1}, and the smooth dependence of the
ground-state manifold on its parameters imply
\[
 \left\|u(t)-e^{-i\theta_0}\lambda_0^{\frac{d-2}{2}}
 W(\lambda_0 x+x_0)\right\|_{\dot H^1}\lesssim e^{-c_0t}.
\]
Define
\[
 U(t,x):=e^{i\theta_0}\lambda_0^{-\frac{d-2}{2}}
 u\left(\frac{t}{\lambda_0^2},\frac{x-x_0}{\lambda_0}\right).
\]
Then $U$ solves \eqref{1.1}, lies on the {upper threshold sheet}, and satisfies
\[
 \|U(t)-W\|_{\dot H^1}\lesssim e^{-c_0\lambda_0^{-2}t}.
\]
Proposition~\ref{2WW} gives $U(t)=W^+(t-T)$ for some $T\in\R$.
Undoing the fixed symmetries proves the corollary.
\end{proof}

\section{{Radial dynamics in the interior of the upper cone}}
\label{sec:radial-subthreshold}

In this section, we prove Theorem~\ref{thm:radial-subthreshold}. We first estimate the nonlinear tail by the exterior Hardy mass $H_R$, and then exclude global radial solutions with a uniformly negative virial functional.

We begin with an estimate for the nonlinear tail.

\begin{lemma}\label{lem:radial-tail}
For every $f\in\dot H^1_{\mathrm{rad}}(\R^d)$ and $R>0$,
\begin{equation}\label{eq:radial-nonlinear-tail}
 \int_{|x|>R}|f(x)|^{2^*}\,dx
 \lesssim\|\nabla f\|_{L_x^2}^{\frac2{d-2}}
 H_R(f)^{\frac{d-1}{d-2}}.
\end{equation}
\end{lemma}

\begin{proof}
By a density argument and the definition of $H_R(f)$, we may   assume that $f\in C_c^\infty(\R^d)$ below. Let
$\omega_{d-1}=|\mathbb S^{d-1}|$. Then polar coordinates give
\[
 \|\nabla f\|_{L_x^2}^2
 =\omega_{d-1}\int_0^\infty|f'(r)|^2r^{d-1}\,dr,
 \qquad
 \left\|\tfrac{f}{|x|}\right\|_{L_x^2}^2
 =\omega_{d-1}\int_0^\infty|f(r)|^2r^{d-3}\,dr.
\]
Set
\[
 r=e^s,\qquad
 v(s)=e^{\frac{d-2}{2}s}f(e^s),\qquad s\in\R.
\]
Recall that $\chi(|x|/R)=1$ for $|x|>R$, we have
\begin{align}
 \int_{\log(R)}^\infty|v(s)|^2\,ds
 &=\int_{R}^\infty r^{d-3}|f(r)|^2\,dr
 =\tfrac1{\omega_{d-1}}
   \int_{|x|>R}\tfrac{|f(x)|^2}{|x|^2}\,dx \notag \\ &\leq\frac{H_R(f)}{\omega_{d-1}}. \label{vbd}
\end{align}
For the exterior potential energy, the same change of variables yields
\begin{align*}
 \int_{|x|>R}|f(x)|^{2^*}\,dx
 =\omega_{d-1}\int_{\log(R)}^\infty
   |v(s)|^{2^*}\,ds.
\end{align*}

We next estimate $v'$:
\[
 v'(s)=\frac{d-2}{2}e^{\frac{d-2}{2}s}f(e^s)
       +e^{\frac d2s}f'(e^s).
\]
Using $|a+b|^2\leq2|a|^2+2|b|^2$ and again changing variables,
we obtain
\begin{align*}
 \int_\R|v'(s)|^2\,ds
 &\leq\frac{(d-2)^2}{2}\int_\R
       e^{(d-2)s}|f(e^s)|^2\,ds
       +2\int_\R e^{ds}|f'(e^s)|^2\,ds\\
 &=\frac{(d-2)^2}{2}\int_0^\infty
       |f(r)|^2r^{d-3}\,dr
       +2\int_0^\infty|f'(r)|^2r^{d-1}\,dr \\
       &\le 4 \int_0^\infty|f'(r)|^2r^{d-1}\,dr=\frac4{\omega_{d-1}}\|\nabla f\|_{L_x^2}^2,
\end{align*}
where the last inequality follows from the sharp Hardy inequality. Similarly, 
\[
 \int_\R|v(s)|^2\,ds
 =\int_0^\infty|f(r)|^2r^{d-3}\,dr
 \leq\frac4{(d-2)^2\omega_{d-1}}\|\nabla f\|_{L_x^2}^2.
\]
In particular, $v\in H^1(\R)$.

Since $f\in C_c^\infty(\R^d)$,
$v(s)\to 0 \;(s\to\infty)$. Note $v\in H^1(\R)$, for every $s\geq\log(R)$ we have
\begin{align*}
 |v(s)|^2
 &=-\int_s^\infty\frac{d}{d\sigma}|v(\sigma)|^2\,d\sigma=-2\Re\int_s^\infty v'(\sigma)\overline{v(\sigma)}\,d\sigma\\
 &\leq2\left(\int_s^\infty|v'(\sigma)|^2\,d\sigma\right)^{1/2}
        \left(\int_s^\infty|v(\sigma)|^2\,d\sigma\right)^{1/2}\\
 &\leq2\|v'\|_{L_s^2(\R)}
        \|v\|_{L_s^2(\log(R),\infty)}.
\end{align*}
Taking the supremum over $s\geq\log(R)$, we obtain
\[
 \|v\|_{L_s^\infty(\log(R),\infty)}^2
 \leq2\|v'\|_{L_s^2(\R)}
        \|v\|_{L_s^2(\log(R),\infty)}.
\]
Together with \eqref{vbd}, this yields
\begin{align*}
 \int_{|x|>R}|f(x)|^{2^*}\,dx
 &=\omega_{d-1}\int_{\log(R)}^\infty
      |v(s)|^{\frac4{d-2}}|v(s)|^2\,ds\\
 &\leq\omega_{d-1}
   \left(\|v\|_{L_s^\infty(\log(R),\infty)}^2\right)^{\frac2{d-2}}
   \|v\|_{L_s^2(\log(R),\infty)}^2\\
 &\leq 2^{\frac2{d-2}}\omega_{d-1}
   \|v'\|_{L_s^2(\R)}^{\frac2{d-2}}
   \|v\|_{L_s^2(\log(R),\infty)}^{\frac{2d-2}{d-2}}\\
 &\lesssim\|\nabla f\|_{L_x^2}^{\frac2{d-2}}
   H_R(f)^{\frac{d-1}{d-2}}.
\end{align*}
This proves \eqref{eq:radial-nonlinear-tail}.
\end{proof}

We next exclude global radial solutions with a uniformly negative virial functional. This argument will also be used in Section~\ref{sec:radial-no-ceiling}.

\begin{lemma}\label{lem:radial-negative-gap}
There is no radial solution $u:[0,\infty)\times\R^d\to\C$ to
\eqref{1.1} for which
\[
 K(u(t)):=\|\nabla u(t)\|_{L_x^2}^2-\|u(t)\|_{L_x^{2^*}}^{2^*}\leq-c_0\qquad(t\geq0)
\]
with a fixed $c_0>0$.
\end{lemma}

\begin{proof}
Suppose that such a solution exists and write $g(t):=-K(u(t))$.
Then the bound $g(t)\geq c_0$, together with \eqref{Ku} yields
\begin{equation}\label{gud}
    \|\nabla u(t)\|_{L_x^2}^2\leq Cg(t),
\end{equation}
where $C$ depends on $E(u)$ and $c_0$. By
\eqref{oneside} and Lemma~\ref{lem:radial-tail},
\begin{align}
 F_R(u(t))
 &\lesssim H_R(u(t))
 +\|\nabla u(t)\|_{L_x^2}^{\frac2{d-2}}
               H_R(u(t))^{\frac{d-1}{d-2}} \notag \\
 &\leq C\left(H_R(u(t))+H_R(u(t))^{\frac{d-1}{d-2}}\right)g(t). \label{FRbd}
\end{align}

For each integer $n\geq0$, as $g(t)\ge c_0$, there exists $t_n\leq c_0^{-1/2}2^n$ such that
\begin{equation}\label{eq:radial-doubling-times}
 \int_0^{t_n}g(t)^{1/2}\,dt=2^n.
\end{equation}
Indeed, $t_n\to\infty$.
Otherwise the integrals in \eqref{eq:radial-doubling-times} would
be bounded ($u\in C([0,T];\dot H^1)$ for every finite $T$) contradicting $2^n\to\infty$.
Set $R_n^2=L2^n$, where $L>1$ to be chosen later.
For $0\leq t\leq t_n$, Lemma~\ref{lem:Hardy-flux} and \eqref{gud} yield
\begin{align*}
 H_{R_n}(u(t))^{1/2}
 &\leq H_{R_n}(u(0))^{1/2}
      +\frac C{R_n^2}\int_0^t\|\nabla u(s)\|_{L_x^2}\,ds\\
 &\leq H_{R_n}(u(0))^{1/2}
      +\frac C{R_n^2}\int_0^{t_n}g(s)^{1/2}\,ds\\
 &=H_{R_n}(u(0))^{1/2}+\frac C L.
\end{align*}
Given $\eta_0>0$, we can choose a fixed $L$ such that $C/L\leq\eta_0^{1/2}/2$.
The dominated convergence theorem shows that
$H_{R_n}(u(0))\to0$. Consequently, for $n$ sufficiently large,
\[
 H_{R_n}(u(t))\leq\eta_0\qquad(0\leq t\leq t_n).
\]
Choosing $\eta_0$ small enough, we get from \eqref{FRbd} and \eqref{eq:localized-virial} that 
\begin{equation}\label{VRn} 
 \frac{d^2}{dt^2}V_{R_n}(u(t))\leq-4g(t), \textit{ for all $t\in[0,t_n]$}.
\end{equation}
Integrating \eqref{VRn} twice and using $V_{R_n}(u(t_n))\geq0$, we obtain
\begin{align*}
 4\int_0^{t_n}(t_n-t)g(t)\,dt
 &\leq V_{R_n}(u(0))+t_n|M_{R_n}(u(0))|.
\end{align*}
By \eqref{eq:virial-initial-smallness}, the two terms on the right
are $o(R_n^4)$ and $t_no(R_n^2)$, respectively. Since
$R_n^2=L2^n$ and $t_n\leq c_0^{-1/2}2^n$, this gives
\begin{equation}\label{eq:radial-double-integral-small}
 \int_0^{t_n}(t_n-t)g(t)\,dt=o(4^n).
\end{equation}
In addition, for $n\geq2$, Cauchy--Schwarz and
\eqref{eq:radial-doubling-times} imply
\begin{align*}
 \int_0^{t_n}(t_n-t)g(t)\,dt
 &\geq(t_n-t_{n-1})\int_{t_{n-2}}^{t_{n-1}}g(t)\,dt\\
 &\geq\frac{t_n-t_{n-1}}{t_{n-1}-t_{n-2}}
       \left(\int_{t_{n-2}}^{t_{n-1}}g(t)^{1/2}\,dt\right)^2\\
 &=4^{n-2}\frac{t_n-t_{n-1}}{t_{n-1}-t_{n-2}}.
\end{align*}
Together with \eqref{eq:radial-double-integral-small}, this yields
\[
 \frac{t_n-t_{n-1}}{t_{n-1}-t_{n-2}}\to 0\;\; (n\to\infty).
\]
In particular, $t_n-t_{n-1}\leq\frac12(t_{n-1}-t_{n-2})$ for
all sufficiently large $n$. Summing this geometric bound shows
that $t_n$ is uniformly bounded, contradicting $t_n\to\infty$. This contradiction completes the proof.
\end{proof}

\begin{proof}[Proof of Theorem~\ref{thm:radial-subthreshold}] Let $I=(T_-,T_+)$.
Energy conservation and Lemma~\ref{lem:subthreshold-branch}
yield
\[
 K(u(t))\leq-2^*\bigl[E(W)-E(u)\bigr]<0
 \qquad(t\in I).
\]
 Lemma~\ref{lem:radial-negative-gap} shows that $T_+<\infty$. Applying the same argument to $\overline{u(-t)}$ gives $T_->-\infty$. Thus the  maximal lifespan of $u$ is finite.
\end{proof}

\section{{Radial dynamics on the upper threshold sheet}}\label{sec:radial-no-ceiling}

In this section, we prove Theorem~\ref{thm:radial-threshold} and show that $W^+$ blows up in finite negative time. We first establish convergence to the ground-state orbit along a sequence of times, and then show that the solution stays sufficiently close to that orbit.

Throughout the section, $u:I\times\R^d\to\C$ is radial and
satisfies \eqref{hyp:bounded}. By Lemma~\ref{lem:threshold-branch},
$\delta(u(t))>0$ and \eqref{eq:Kdelta} hold. The spatial modulation
parameter is zero for radial symmetry.

We begin by controlling the sign of the virial derivative by the exterior Hardy mass.

\begin{lemma}\label{lem:radcoer0}
There exists $\eta_0>0$ such that, for every $R>0$ and $t\in I$,
\begin{equation}\label{radialcoer}
 H_R(u(t))\leq\eta_0
 \quad\Longrightarrow\quad
 \frac d{dt}M_R(u(t))\leq-\frac8{d-2}\delta(u(t)).
\end{equation}
The constant $\eta_0$ depends only on the dimension and fixed
cutoff functions.
\end{lemma}

\begin{proof}
By \eqref{thlocalv}, it suffices to prove for every $R>0$ and $t\in I$,
\[
  H_R(u(t))\leq\eta_0
 \quad\Longrightarrow\quad F_R(u(t))\leq\frac8{d-2}\delta(u(t)).
\]
 Fix $0<\delta_1<\delta_0$
sufficiently small. Suppose first that $0<\delta(u(t))<\delta_1$.
Let
\[
 \widetilde W(x):=e^{-i\theta(t)}\lambda(t)^{\frac{d-2}{2}}
                      W(\lambda(t)x),
\]
where $\theta$ and $\lambda$ are given by Lemma~\ref{modula}.
Then $\|u(t)-\widetilde W\|_{\dot H^1}\lesssim\delta(u(t))$.
The explicit formula for $W$ gives, uniformly in $\lambda(t)$,
\[
 |\nabla\widetilde W(x)|\lesssim\frac{|\widetilde W(x)|}{|x|},
 \qquad |\widetilde W(x)|^{\frac4{d-2}}\lesssim|x|^{-2}.
\]
This together with $\chi(|x|/R)=1$ for $|x|\geq R$ yield
\begin{align*}
 \int_{|x|>R}\left(|\nabla\widetilde W|^2
     +|\widetilde W|^{2^*}
     +\frac{|\widetilde W|^2}{|x|^2}\right)\,dx
 &\lesssim H_R(\widetilde W).
\end{align*}
From Hardy's inequality and the
modulation estimate (Lemma \ref{modula}), we have 
 $$ H_R(\widetilde W)\leq2H_R(u(t))+C\delta(u(t))^2.$$
Applying \eqref{VED}
with $h=u(t)-\widetilde W$, we obtain
\begin{align*}
 |F_R(u(t))|
 &\lesssim H_R(\widetilde W)^{1/2}\delta(u(t))+\delta(u(t))^2 +H_R(\widetilde W)^{\frac{d+2}{2d}}\delta(u(t))
          +\delta(u(t))^{2^*}\\
&\lesssim\bigl(H_R(u(t))^{1/2}+\delta(u(t))\bigr)\delta(u(t)),
\end{align*}
provided that $H_R(u(t))$ and $\delta_1$ are sufficiently small.

If $\delta(u(t))\geq\delta_1$,  
\eqref{oneside} and Lemma~\ref{lem:radial-tail} imply
\begin{align*}
 F_R(u(t))
 &\lesssim H_R(u(t))
  +\|\nabla u(t)\|_{L_x^2}^{\frac2{d-2}}
          H_R(u(t))^{\frac{d-1}{d-2}}\\
 &\lesssim H_R(u(t))
  +\bigl(\|\nabla W\|_{L_x^2}^2+\delta(u(t))\bigr)^{\frac1{d-2}}
          H_R(u(t))^{\frac{d-1}{d-2}}\\
          &\leq C_{\delta_1}
 \left(H_R(u(t))+H_R(u(t))^{\frac{d-1}{d-2}}\right)\delta(u(t)).
\end{align*}

The conclusion follows by first taking $\delta_1$ sufficiently small and then $\eta_0$.
\end{proof}

We next prove that $\delta(u(t))$ converges to zero along a sequence of times.

\begin{proposition}\label{prads}
Let $u:[0,\infty)\times\R^d\to\C$ be a radial solution to \eqref{1.1} satisfying
\eqref{hyp:bounded}. Then
\begin{equation} 
 \liminf_{t\to\infty}\delta(u(t))=0.
\end{equation}
\end{proposition}

\begin{proof}
We argue by contradiction. Assume after a time translation  that  $\delta(u(t))\geq c_0>0$
for all $t\geq0$. By \eqref{eq:Kdelta}, we have
\[
K(u(t))=-\frac2{d-2}\delta(u(t))\leq-\frac{2c_0}{d-2}\qquad(t\geq0),
\]
contradicting Lemma~\ref{lem:radial-negative-gap}. This proposition is proved.
\end{proof}

{We next show that a radial solution on the upper threshold sheet which is close to the ground-state manifold at both endpoints of an interval must remain close to that manifold throughout that interval.}

\begin{lemma}\label{lem:radial-no-return}
Let $u:[0,\infty)\times\R^d\to\C$ be a radial solution satisfying
\eqref{hyp:bounded}, and $0<\eta<\delta_0$. There exists
$\varepsilon=\varepsilon(\eta)>0$ such that the following holds.
Suppose that there exist 
$t_1<t_2$ such that
\[
 \delta(u(t_1))+\delta(u(t_2))\leq\varepsilon.
\]
Then
\begin{equation}\label{unib0}
 \sup_{t\in[t_1,t_2]}\delta(u(t))<\eta.
\end{equation}
\end{lemma}

\begin{proof}
Take $0<\varepsilon<\eta/2$ sufficiently small. Suppose that
\eqref{unib0} fails. By the symmetries of
the equation, we may assume that $[t_1,t_2]=[0,T]$ and 
$\lambda(0)=1$, $\theta(0)=0$. Let $s\in(0,T)$ be the first time
such that $\delta(u(s))=\eta$, then
\begin{equation}\label{lowerbd1}
 \int_0^T\delta(u(t))\,dt\gtrsim\eta.
\end{equation}
Since $\eta<\delta_0$, it suffices to consider the case where
$\int_0^s\delta(u(t))\,dt$ is smaller than a sufficiently small constant. By \eqref{Mod2}, on $[0,s]$, we have
\[
 \left|\frac{d}{dt}\lambda(t)^{-2}\right|\lesssim\delta(u(t)).
\]
Note that $\lambda(0)=1$, integration gives $\lambda(t)^2\leq2$ on
$[0,s]$. Applying \eqref{equivq1} and \eqref{Mod2}, we obtain
\begin{align*}
 \eta&\lesssim|\alpha(s)|
 \leq|\alpha(0)|+\int_0^s|\alpha'(t)|\,dt
 \lesssim\varepsilon+\int_0^s\delta(u(t))\,dt.
\end{align*}
Taking $\varepsilon$ sufficiently small relative to $\eta$ proves
\eqref{lowerbd1}.

Set
\begin{equation}\label{R2}
 R^2=L\int_0^T\delta(u(t))\,dt,
\end{equation}
where $L=L(\eta)>1$ is chosen later.
By \eqref{lowerbd1}, $R^2\gtrsim L\eta$.
Since $H_R(W)\to0$ as $R\to\infty$, choosing $L$ sufficiently
large and then $\varepsilon$ sufficiently small gives
\[
 H_R(u(0))\leq2H_R(W)+C\|u(0)-W\|_{\dot H^1}^2
 \leq2H_R(W)+C\varepsilon^2\leq\frac12\eta_0,
\]
where $\eta_0$ is as in Lemma~\ref{lem:radcoer0}.
Moreover, for $0\leq t\leq T$, \eqref{refHardy} gives
\begin{align*}
 H_R(u(t))
 &\leq H_R(u(0))+\frac{C}{R^2}\int_0^t\delta(u(\tau))\,d\tau \leq\frac12\eta_0+\frac{C}{L}\leq\eta_0,
\end{align*}
provided $L$ is sufficiently large depending on $\eta$.
Thus, \eqref{radialcoer} holds on $[0,T]$.
From integrating \eqref{radialcoer} and  \eqref{rac0} in Lemma~\ref{lem:virial-estimates}, \eqref{R2} then yields
\begin{align*}
 \int_0^T\delta(u(t))\,dt
 &\lesssim M_R(u(0))-M_R(u(T)) \lesssim R^2 (\delta(u(0))+\delta(u(T)) )\\ 
 &\lesssim L\varepsilon\int_0^T\delta(u(t))\,dt.
\end{align*}
We get a contradiction by taking $\varepsilon$ sufficiently small, and \eqref{unib0} holds.
\end{proof}

\begin{proof}[Proof of Theorem~\ref{thm:radial-threshold}]
Let $I=(T_-,T_+)$. If $I$ is bounded, there is nothing to prove. Otherwise, after
applying time reversal if necessary, we assume that $T_+=\infty$.
Fix $0<\eta<\min\{\eta_*,\delta_0\}$, where $\eta_*$ is as in
Corollary~\ref{t2.3}, and let $\varepsilon$ be given by
Lemma~\ref{lem:radial-no-return}. By
Proposition~\ref{prads}, there exist $n_0\in \mathbb{N}^+$ and
an increasing sequence $t_n\to\infty$ such that
$\delta(u(t_n))\leq\varepsilon/2$ for all $n\ge n_0$. Applying
Lemma~\ref{lem:radial-no-return} on $[t_{n_0},t_n]$ and letting
$n\to\infty$, we obtain
\[
 \sup_{t\geq t_{n_0}}\delta(u(t))\leq\eta<\eta_*.
\]
 After a time translation, Corollary~\ref{t2.3} implies that $u=W^+$ modulo
symmetries. This completes the proof.
\end{proof}

As a consequence, the special solution $W^+$ has a finite backward lifespan in dimensions $d\ge 3$. This removes the additional finite-mass condition in \cite{DM08}.

\begin{corollary}\label{cor:Wplus-finite-backward}
Let $d\geq3$, and let $I$ be the maximal lifespan of $W^+$. Then
$\inf I> -\infty$.
\end{corollary}

\begin{proof}
Suppose otherwise that $\inf I=-\infty$. Since $W^+$ is global forward in
time, we have $I=\R$. Applying
Proposition~\ref{prads} to
$\overline{W^+(-t)}$, we find $s_n\to-\infty$ such that
$\delta(W^+(s_n))\to0$. {Moreover, by Proposition \ref{2WW}, $\delta(W^+(n))\to0$.} 
Fix $0<\eta<\delta_0$, and let $\varepsilon=\varepsilon(\eta)>0$
be given by Lemma~\ref{lem:radial-no-return}. The preceding limits
give $\delta(W^+(s_n))+\delta(W^+(n))\leq\varepsilon$ for all
sufficiently large $n$. Applying Lemma~\ref{lem:radial-no-return} on $[s_n,n]$, we obtain
\[
 \sup_{t\in[s_n,n]}\delta(W^+(t))<\eta.
\]
Letting $n\to\infty$ and then $\eta\downarrow0$, we find
$\delta(W^+(t))=0$ for every $t\in\R$, contradicting the
{upper-sheet} condition \eqref{hyp:bounded}. This proves the corollary.
\end{proof}

\section{{Nonradial dynamics in the upper cone}}\label{sec:general}

{In this section, we prove Theorems~\ref{thm:general-subthreshold}
and \ref{thm:general-threshold} for non-radial data. The
finite-variance classification applies in every dimension
$d\geq3$, while the results under an explicit kinetic-energy ceiling, without spatial localization, are restricted
to $d\geq5$. We first treat finite-variance solutions on the upper threshold sheet. We then develop the profile and compactness arguments for the kinetic energy-bounded regime and prove Theorem~\ref{thm:general-subthreshold}(ii) and Theorem~\ref{thm:general-threshold}(ii) in Sections \ref{subsec:Thmsubthreshold} and \ref{subsec:Thmthreshold}. }

\subsection{Finite-variance solutions}\label{reduction}

We treat the threshold energy case: $E(u_0)=E(W)$. We begin with a weighted analogue of the inequality of Banica \cite[Lemma 2.1]{Bonica}. \footnote{See \cite{DR10} for the energy sub-critical case.}
\begin{lemma}
\label{lem:banica}
Let $u$ be a {solution on the upper threshold sheet} 
and let
$\phi\in C^1(\R^d;\mathbb R)$. Then
\begin{equation}
\left| \Im\int_{\mathbb{R}^{d}}\nabla \phi \cdot \nabla u(t)
\overline{u(t)}dx\right|^{2}
\lesssim \delta(u(t))^{2}\int_{\mathbb{R}^{d}}
|\nabla\phi |^{2}|u(t)|^{2}\, dx,\label{B}
\end{equation}%
where the implicit constant does not depend on $u(t)$ and $\phi$.
\end{lemma}
\begin{proof}
The sharp Sobolev inequality in Lemma~\ref{GNsharpz} yields
\[
\Vert e^{i\lambda \phi }u(t)\Vert _{\dot{H}^{1}}^{2}\Vert W\Vert
_{2^{\ast }}^{2}\geq \Vert W\Vert _{\dot{H}^{1}}^{2}\Vert u(t)\Vert
_{2^{\ast }}^{2},
\]
 for every real number $\lambda$ and for any real-valued function $\phi$.
Equivalently,%
\begin{align*}
\lambda ^{2}\int_{\mathbb{R}^{d}}|\nabla\phi |^{2}|u(t)|^{2}dx &+2\lambda \Im
\int_{\mathbb{R}^{d}}\nabla \phi \cdot \nabla u(t)\overline{u(t)}dx \\
 &+\Vert
u(t)\Vert _{\dot{H}^{1}}^{2}-\frac{\Vert W\Vert _{\dot{H}^{1}}^{2}}{%
\Vert W\Vert _{2^{\ast }}^{2}}\Vert u(t)\Vert _{2^{\ast }}^{2}\geq 0. 
\end{align*}
Hence, the discriminant of the equation in $\lambda$ must be non-positive and  
\begin{align}\nonumber
&\left\vert \Im\int_{\mathbb{R}^{d}}\nabla \phi \cdot \nabla u(t)%
\overline{u(t)}dx\right\vert ^{2}\\
\leq& \Big(\Vert u(t)\Vert _{\dot{H}^{1}}^{2}-\frac{\Vert W\Vert _{\dot{H}^{1}}^{2}}{\Vert W\Vert
_{2^{\ast }}^{2}}\Vert u(t)\Vert _{2^{\ast }}^{2}\Big)\int_{\mathbb{R}^{d}}|\nabla\phi
|^{2}|u(t)|^{2}dx.  \label{B1}
\end{align}%
From $\delta(u(t))=\|u(t)\|_{\dot H^1}^2-\|W\|_{\dot H^1}^2$ and $E(u(t))=E(W)$, we have
\begin{align*}
\Big| \Vert u(t)\Vert _{\dot{H}^{1}}^{2}-\frac{\Vert W\Vert _{\dot{H}^{1}}^{2}}{\Vert W\Vert _{2^{\ast }}^{2}}\Vert u(t)\Vert _{2^{\ast
}}^{2}\Big|
=&\Big|\Vert u(t)\Vert _{\dot{H}^{1}}^{2}-\Vert W\Vert _{\dot{H}^{1}}^{2}\Big[1+\frac{2^{\ast }\delta(u(t))}{2\Vert W\Vert _{\dot{H}^{1}}^{2}}\Big]^{\frac{2}{2^{\ast }}}\Big| \\
=&\Big| \delta(u(t))-\Vert W\Vert _{\dot{H}^{1}}^{2}\Big(\Big[1+\frac{2^{\ast }\delta(u(t))}{2\Vert W\Vert _{\dot{H}^{1}}^{2}}\Big]^{\frac{2}{2^{\ast }}}-1\Big)\Big| \\
\leq &\frac1{(d-2)\|W\|_{\dot{H}^1}^2}\delta(u(t))^{2},
\end{align*}%
where the last inequality follows from
$$\Big|\delta-C_1\Big(\big(1+\frac{2^\ast\delta}{2C_1}\big)^\frac{2}{2^\ast}-1\Big)\Big|\leq \frac{1}{(d-2)C_1}\delta^2,\quad \forall\;C_1>0,\delta\geq0.$$
Plugging this into \eqref{B1} implies \eqref{B}. This completes the proof.
\end{proof}

{Lemma \ref{lem:banica} yields the following properties of finite-variance solutions on the upper threshold sheet.} 
\begin{lemma}\label{unib}
Let $u$ be a {solution on the upper threshold sheet}  with $u|_{t=0}=u_0\in \Sigma$.
If $u$ is global forward in time, then:
\begin{enumerate}
\item[(i)]
\begin{equation}\label{finiteb}
    \sup_{t\geq0}\|xu(t)\|_{L_x^2}<\infty, \qquad \|u(t)\|_{L_x^2}=\|u_0\|_{L_x^2}<\infty;
\end{equation}

\item[(ii)] 
There exists $c_0>0$ such that
 \begin{equation}\label{expo}
    \int_t^{\infty}\delta(u(\tau))\, d\tau\lesssim e^{-c_0t},
\end{equation}
and there exists a sequence $t_n\to\infty$ such that $\delta(u(t_n))\to 0$. 
\end{enumerate}
\end{lemma}
\begin{proof}
    Define $V(t)=\int_{\mathbb{R}^{d}}|x|^{2}|u(t)|^{2}dx$. Direct computation gives 
\begin{align*}
V^{\prime }(t)=&4\Im\int_{\mathbb{R}^{d}}x\cdot \nabla u(t)
\overline{u(t)}\, dx; \\
V^{\prime \prime }(t) =&8\int_{\mathbb{R}^{d}}|\nabla u(t)|^{2}-|u(t)|^{2^*}\,dx=-\frac{16}{d-2}\delta(u(t)).
\end{align*}
Since $\delta(u(t))>0$, we have $V''(t)<0$ and
\begin{equation}\label{V1}
V^{\prime }(t)>0,\;\forall\ t\in [0,\infty ).
\end{equation}
Indeed, if there exists $t_0>0$ such that $V'(t_0)\leq0$, then
\[
V^{\prime }(t)\leq V^{\prime }(t_{0}+1)<0,\;\forall\ t\in [
t_{0}+1,\infty ),
\]
which shows $V(t)\to-\infty$ as $t\to\infty$, contradicting $V(t)\ge 0$. 

Applying Lemma~\ref{lem:banica} with $\phi=|x|^2$ gives
\begin{equation}\label{bon}
V^{\prime }(t)^{2}\leq CV^{\prime \prime }(t)^{2}V(t).
\end{equation}
Together with \eqref{V1}, this yields
\[
\frac{V^{\prime }(t)}{\sqrt{V(t)}}\leq -CV^{\prime \prime }(t).
\]
Integrating this gives
\[
\sqrt{V(t)}-\sqrt{V(0)}\leq -C(V^{\prime }(t)-V^{\prime }(0))\leq CV^{\prime
}(0)<\infty.
\]
Together with Hardy's inequality and mass conservation, this yields part (i).
From \eqref{finiteb} and \eqref{bon}, we obtain
\[
V^{\prime }(t)\leq -CV^{\prime \prime }(t).
\]%
Hence, $V'(t)\lesssim e^{-c_0t}$ for some $c_0>0$.  Integrating above inequality gives
\[
V^{\prime }(t)=-\int_{t}^{\infty }V^{\prime \prime }(\tau )d\tau =\frac{16}{%
d-2}\int_{t}^{\infty }\delta(u(\tau ))d\tau \leq Ce^{-c_0t}.
\]
This proves part (ii) of this lemma.
\end{proof}


\begin{proof}[{\bf Proof of Theorem~\ref{thm:general-threshold}(i)}]
It suffices to show that $I$ is unbounded implies $u=W^+$ modulo symmetries. Applying the time-reversal symmetry $u(t,x)\mapsto\overline{u(-t,x)}$ if necessary, $\sup I=\infty$.  By Lemma~\ref{unib}(ii), there exists
an increasing sequence $t_n\to\infty$ such that
$\delta(u(t_n))\to0$. Let $\eta_*$ be as in Corollary~\ref{t2.3}, and we divide the proof into two cases:
\begin{enumerate}
\item[(i)] There exists $t_0>0$ such that
\[
\sup_{t\in[t_0,\infty)}\delta(u(t)) \leq \eta_*.
\]
\item[(ii)] There exists an increasing sequence $t_n^-\to\infty$ such that
\[
\delta(u(t_n^-))>\eta_*\qtq{for all $n$.}
\]
\end{enumerate}

In case (i), by Corollary~\ref{t2.3}, after a time translation, we get
$u=W^+$ modulo symmetries.

Suppose that case (ii) holds.  After passing to a subsequence, we may arrange the times as displayed below.
\[
\delta(u(t_n)) < \eta_*,\qquad t_n^- < t_n < t_{n+1}^-.
\]
Set
\[
t_n^+ = \inf\bigl\{t\in[t_n^-,t_n]:
\sup_{\tau\in[t,t_n]} \delta(u(\tau))<\eta_*\bigr\},
\]
then
\[
t_n^+\to\infty,\qquad\delta(u(t_n^+))=\eta_*.
\]
Denote $I_n=[t_n^+,t_n]$.  By Lemma~\ref{modula}, 
there exist 
\begin{equation*}
\theta : I_n\rightarrow \R,\quad x:  I_n\rightarrow \mathbb{R}^d ,\quad \lambda: I_n\rightarrow \mathbb{R}_+, \quad \alpha: I_n\rightarrow \mathbb{R}
\end{equation*}
such that on $I_n$,
\begin{equation}\label{decompz2}
v(t):=\alpha(t)W+\tilde{u}(t)=u_{[\theta(t),\lambda(t),x(t)]}(t)-W,
\end{equation}
\begin{align}
&|\alpha(t)|\sim \|v(t)\|_{\dot{H}^1}\sim \|\tilde{u}(t)\|_{\dot{H}^1}\sim \delta(u(t)),\label{mo12}\\
&\bigl|{x^{\prime}(t)}-\tfrac{\lambda^{\prime}(t)}{\lambda(t)}x(t)\bigr|+|\alpha^{\prime}(t)|+|\theta^{\prime}(t)|+\bigl|\tfrac{\lambda^{\prime}(t)}{\lambda(t)}\bigr|\leq C\lambda(t)^2 \delta(u(t)). \label{mo22}
\end{align}
The $L^2$ regularity gained from finite variance \eqref{finiteb} shows that
\begin{equation}\label{moll}
    \lambda(t)\gtrsim 1.
\end{equation}
Indeed, the mass conservation and $\|v(t)\|_{L_{x}^{2}(|x|\leq 1)} \lesssim \|v\|_{\dot{H}^1} \lesssim \delta_0$ yield 
\begin{align*}
M(u_0)&=\lambda(t)^{-2}\|W+v(t)\|_{L_x^2}^2\geq \lambda(t)^{-2}\|W+v(t)\|_{L_x^2(|x|\le 1)}^2 \\
 & \ge \lambda(t)^{-2} \bigl(\|W\|_{L_x^2(|x|\le 1)}-\|v(t)\|_{L_x^2(|x|\le 1)}\bigr)^2
 \gtrsim\lambda(t)^{-2}.
 \end{align*}

Next, we claim that
\begin{equation}\label{claimfinal}
    \sup_n \lambda(t_n^+)<\infty.
\end{equation}
Applying Theorem~\ref{Ret2.3} on $I_n$ gives
\begin{equation}\label{bbz}
    \lambda(t_n)\sim\lambda(t_n^+)
\end{equation}
uniformly in $n$, and 
$\lambda(t)\sim\lambda(t_n^+)$
throughout $I_n$. Set
\[
 \widetilde W_n(x):=e^{-i\theta(t_n)}\lambda(t_n)^{\frac{d-2}{2}}
 W(\lambda(t_n)x+x(t_n)).
\]
In view of \eqref{mo12} and $\delta(u(t_n))\to0$, this implies
\begin{equation}\label{equ:dotH1zero}
    \lim_{n\to\infty}\|u(t_n)- \widetilde W_n\|_{\dot H^1}= 0.
\end{equation}
By the conservation of mass, the boundedness of $\delta(u(t_n))$,  \eqref{moll} and $W\in L_x^{\frac{6d}{3d-4}}(\R^d)$, we have
\[
\sup_{n}\|u(t_n)- \widetilde W_n\|_{L_x^{\frac{6d}{3d-4}}}<\infty.
\]
Interpolating with \eqref{equ:dotH1zero} gives
\begin{equation}\label{final2}
    \lim_{n\to\infty}\|u(t_n)- \widetilde W_n\|_{L_x^{\frac{4d}{2d-3}}}= 0.
\end{equation}
If \eqref{claimfinal} fails, then, after passing to a subsequence,
\[
\lim_{n\to\infty}\lambda(t_n^+)=\infty.
\]
From \eqref{bbz} and \eqref{final2}, we have
\begin{equation}\label{final3}
   \lim_{n\to\infty} \|\widetilde W_n\|_{L_x^{\frac{4d}{2d-3}}} =0, \qquad  \lim_{n\to\infty} \|u(t_n)\|_{L_x^{\frac{4d}{2d-3}}} =0.
\end{equation}
For $R>0$, H\"older's inequality implies
\begin{align*}
\int_{\R^d}|u(t_n,x)|^2\,dx
&\leq \int_{|x|\leq R}|u(t_n,x)|^2\,dx
 +\int_{|x|\geq R}|u(t_n,x)|^2\,dx\notag\\
&\lesssim R^{\frac32}
 \bigl(\int_{|x|\leq R}|u(t_n,x)|^{\frac{4d}{2d-3}}\,dx\bigr)^{\frac{2d-3}{2d}}+\int_{|x|\geq R}\tfrac{|xu(t_n,x)|^2}{R^2}\,dx\notag\\
&\lesssim R^{\frac32}
 \bigl(\int_{\R^d}|u(t_n,x)|^{\frac{4d}{2d-3}}\,dx\bigr)^{\frac{2d-3}{2d}}+\frac1{R^2}\int_{\R^d}|xu(t_n,x)|^2\,dx.
\end{align*}
From  Lemma~\ref{unib}(i) and \eqref{final3},  $u(t,x)$ must have zero mass by first taking $n\to\infty$ and then $R\to\infty$. This contradiction  shows \eqref{claimfinal}. Consequently, 
 \[
\lambda(t)\sim\lambda(t_n)\sim\lambda(t_n^{+})\sim 1,\quad \forall t\in I_n.
 \]
Then, as $n\to\infty$, \eqref{mo12},
\eqref{mo22}, and \eqref{expo} yield
\begin{align*}
 \eta_*&\lesssim|\alpha(t_n^+)|\leq|\alpha(t_n)|+\int_{t_n^+}^{t_n}|\alpha'(t)|\,dt\\
 &\lesssim\delta(u(t_n))+\int_{t_n^+}^{\infty}\delta(u(t))\,dt
 \lesssim\delta(u(t_n))+e^{-c_0t_n^+}\to 0.
\end{align*}
This contradiction excludes case (ii) and concludes the proof.
\end{proof}

\subsection{Concentration compactness and almost periodicity}
\label{sec:sequential}\label{sec:bounded-proof}

Throughout the remainder of this section, we assume $d\geq5$. Threshold
solutions satisfy \eqref{hyp:bounded} and, unless otherwise stated,
the bound \eqref{hyp:ceiling}.

We use the local theory and stability in Theorem~\ref{T:local} and Lemma~\ref{pertu}.

For $\lambda_0>0$ and $x_0\in\R^d$, define
\[
 (g_{\lambda_0,x_0}f)(x)
 :=\lambda_0^{-\frac{d-2}{2}}
 f\left(\frac{x-x_0}{\lambda_0}\right).
\]
We denote by $G$ the group of all such transformations.

\begin{lemma}[Linear profile decomposition, \protect\cite{KV1,MMMZ25}]\label{lem:linear-profiles}
Let $\{u_n\}$ be bounded in $\dot H^1(\R^d)$.  After passing to a
subsequence, there exist $\{\phi^j\}_{j\geq1}\subset\dot H^1(\R^d)$,
$g_n^j=g_{\lambda_n^j,x_n^j}\in G$, $t_n^j\in\R$, and
$w_n^J\in\dot H^1(\mathbb{R}^d)$ such that
\begin{equation}\label{eq:full-lpd}
 u_n=\sum_{j=1}^Jg_n^je^{it_n^j\Delta}\phi^j+w_n^J
\end{equation}
for all $J\geq1$.  After modifying the profiles, $t_n^j\equiv0$ or
$t_n^j\to\pm\infty$ for every $j$.  Moreover,
\begin{equation}\label{eq:linear-remainder-small}
 \lim_{J\to\infty}\limsup_{n\to\infty}
 S_\R(e^{it\Delta}w_n^J)=0.
\end{equation}
For $j\neq k$, the parameters are asymptotically orthogonal:
\begin{align}\label{eq:parameter-orthogonality}
 \frac{\lambda_n^j}{\lambda_n^k}
 +\frac{\lambda_n^k}{\lambda_n^j}
 +\frac{|x_n^j-x_n^k|^2}{\lambda_n^j\lambda_n^k}
 +\frac{|t_n^j(\lambda_n^j)^2-t_n^k(\lambda_n^k)^2|}
 {\lambda_n^j\lambda_n^k}\to\infty.
\end{align}
For each fixed $J$,
\begin{align}
 \|\nabla u_n\|_{L_x^2}^2
 &=\sum_{j=1}^J\|\nabla\phi^j\|_{L_x^2}^2
   +\|\nabla w_n^J\|_{L_x^2}^2+o_n(1),\label{eq:kinetic-decoupling}\\
 \|u_n\|_{L_x^{2^*}}^{2^*}
&=\sum_{j=1}^J\|e^{it_n^j\Delta}\phi^j\|_{L_x^{2^*}}^{2^*}
   +\|w_n^J\|_{L_x^{2^*}}^{2^*}+o_n(1).
   \label{eq:potential-decoupling}
\end{align}
\end{lemma}

We now define the nonlinear profiles
$U^j:I^j\times\R^d\to\C$ associated with $\phi^j$ and $t_n^j$ as
follows:
\begin{itemize}
\item If $t_n^j\equiv0$, then $U^j$ is the maximal-lifespan solution
with initial data $U^j(0)=\phi^j$.
\item If $t_n^j\to\infty$, then $U^j$ is the maximal-lifespan solution
that scatters forward in time to $e^{it\Delta}\phi^j$.
\item If $t_n^j\to-\infty$, then $U^j$ is the maximal-lifespan solution
that scatters backward in time to $e^{it\Delta}\phi^j$.
\end{itemize}
For each $j,n\geq1$, set
\begin{equation}\label{eq:rescaled-nonlinear-profile}
 U_n^j(t):=g_n^jU^j\bigl(t_n^j+\tfrac{t}{(\lambda_n^j)^2}\bigr),\quad
 I_n^j:=\bigl\{t\in\R:
 t_n^j+\tfrac{t}{(\lambda_n^j)^2}\in I^j\bigr\}.
\end{equation}

\begin{definition}[Almost periodic solution]
A solution $v:I\times\R^d\to\C$ to \eqref{1.1} is almost periodic modulo {the scaling-translation group $G$} if
there exist $N:I\to(0,\infty)$, $x:I\to\R^d$, and
$C:(0,\infty)\to(0,\infty)$ such that, for every $t\in I$ and
$\eta>0$,
\begin{align*}
 &\int_{|x-x(t)|>C(\eta)/N(t)}|\nabla v(t,x)|^2\,dx
 +\int_{|\xi|>C(\eta)N(t)}|\xi|^2|\widehat v(t,\xi)|^2\,d\xi
 <\eta.
\end{align*}
By Sobolev embedding, $C(\eta)$ may also be chosen so that
\[
 \int_{|x-x(t)|>C(\eta)/N(t)}
 |v(t,x)|^{2^*}\,dx<\eta.
\]
Equivalently, the orbit
 $\bigl\{g_{1/N(t),x(t)}^{-1}v(t):t\in I\bigr\}$
is precompact in $\dot H^1(\R^d)$.  In particular,
$\sup_{t\in I}\|\nabla v(t)\|_{L_x^2}<\infty$.
\end{definition}

\begin{lemma}[Exclusion of finite lifespan endpoints]\label{lem:finite-ap}
Every maximal-lifespan almost-periodic solution is global. 
\end{lemma}

\begin{proof} The proof is similar to \cite{kenig-merle1, DM08,MMMZ25}.
Let $v$ be an almost-periodic solution with maximal-lifespan $I=(T_-,T_+)$. It remains to show $T_+=\infty$; the proof that $T_-=-\infty$ is identical. Suppose that $T_+<\infty$.  We first claim that
\begin{equation}\label{eq:frequency-at-endpoint}
 \liminf_{t\uparrow T_+}N(t)=\infty.
\end{equation}
If \eqref{eq:frequency-at-endpoint} fails, then there exists
$t_n\uparrow T_+$ such that $\lim_{n\to\infty}N(t_n)<\infty$.  Define
\[
 v_n(t,x):=N(t_n)^{-\frac{d-2}{2}}
 v\bigl(t_n+\tfrac{t}{N(t_n)^2},
 x(t_n)+\tfrac{x}{N(t_n)}\bigr).
\]
By definition, $\{v_n(0)\}$ is precompact in $\dot H^1(\R^d)$. Passing to a
subsequence, we may assume that $v_n(0)\to v_0$ in $\dot H^1(\R^d)$.  Let
$V$ be the solution with $V(0)=v_0$.  The local theory and stability give a compact
interval $J\ni0$ such that
\[
 \limsup_{n\to\infty}S_J(v_n)<\infty.
\]
Since $t_n\uparrow T_+$ and $\lim_{n\to\infty}N(t_n)<\infty$, stability extends
$v$ beyond $T_+$. This contradiction proves
\eqref{eq:frequency-at-endpoint}.

Let $\chi_R$ be a smooth cutoff supported on $\{|x|\leq R\}$.  We shall prove
\begin{equation}\label{eq:local-mass-vanishes}
 \lim_{t\uparrow T_+}\|v(t)\chi_R\|_{L_x^2}=0,\;\forall\;R>0.
\end{equation}
Indeed, for $0<\eta\ll1$, H\"older's inequality and Sobolev embedding
give
\begin{align*}
 \|v(t)\chi_R\|_{L_x^2}
 &\leq \|v(t)\|_{L_x^2(|x-x(t)|\leq\eta R)}
 +\|v(t)\chi_R\|_{L_x^2(|x-x(t)|>\eta R)}\\
 &\lesssim \eta R\|v(t)\|_{L_x^{2^*}}
 +R\|v(t)\|_{L_x^{2^*}(|x-x(t)|>\eta R)}.
\end{align*}
Choosing $\eta$ sufficiently small and using almost periodicity with \eqref{eq:frequency-at-endpoint}, we obtain \eqref{eq:local-mass-vanishes}. From H\"older and Hardy's inequality, we have
\[
 \left|\frac{d}{dt}\|v(t)\chi_R\|_{L_x^2}^2\right|
 \lesssim  \| v(t)\|_{\dot H^1}
 \|v(t)\|_{L_x^{2^*}}
 \tfrac{\| \nabla\chi_R\|_{L_x^d}}{R}
 \lesssim \| v(t)\|_{\dot H^1}^2\|\nabla \chi\|_{L_x^d}\lesssim 1.
\]
This uniform upper bound yields that, for $C>0$ and $t<t_2<T_+$,
\[
 \|v(t)\chi_R\|_{L_x^2}^2
 \leq \|v(t_2)\chi_R\|_{L_x^2}^2+C(t_2-t).
\]
From \eqref{eq:local-mass-vanishes}, taking $t_2\uparrow T_+$ yields
\[
 \|v(t)\chi_R\|_{L_x^2}^2\leq C(T_+-t).
\]
Taking $R\to\infty$ shows that $v(t)\in L_x^2(\R^d)$ and
$
 \|v(t)\|_{L_x^2}^2\leq C(T_+-t).
$
Mass conservation yields $v\equiv0$. This contradiction completes the proof. 
\end{proof}

We next establish positivity of the profile energies. Recall that $\kappa_d=\bigl(\frac d{d-2}\bigr)^{\frac{d-2}{4}}$.
\begin{lemma}[{Coercivity under the kinetic-energy ceiling}]\label{lem:ceiling-coercivity}
Let $0<M<\kappa_d\|\nabla W\|_{L_x^2}$. Then, there exists $c_M>0$ such that for any $0\neq f\in \dot H^1(\R^d)$
\begin{equation}\label{eq:ceiling-coercivity}
 \|\nabla f\|_{L_x^2}\leq M
 \quad\Longrightarrow\quad
 E(f)\geq c_M\|\nabla f\|_{L_x^2}^2.
\end{equation}
If $\{u_n\}$ is bounded in $\dot H^1(\R^d)$, $\limsup_n\|\nabla u_n\|_{L_x^2}\leq M$, and
\begin{equation}\label{eq:lpd}
 u_n=\sum\nolimits_{j=1}^J g_n^j e^{it_n^j\Delta}\phi^j+w_n^J
\end{equation}
is a linear profile decomposition with associated nonlinear profiles
$U^j$, then
\begin{equation}\label{eq:energy-decoupling}
 E(u_n)=\sum_{j=1}^J E(U^j)+E(w_n^J)+o_n(1),
\end{equation}
where
\[
 E(U^j)=\lim_{n\to\infty}E(e^{it_n^j\Delta}\phi^j)
 =\begin{cases}
 E(\phi^j),&t_n^j\equiv0,\\
 \frac12\|\nabla\phi^j\|_{L_x^2}^2,&|t_n^j|\to\infty.
 \end{cases}
\]
In particular, $E(U^j)\geq0$ for every $j$, and
$\liminf_{n }E(w_n^J)\geq0$ for every $J$.
If, in addition, $\limsup _{n }E(u_n)\leq E(W)$, then
\begin{equation}\label{eq:at-most-one-supercritical-profile}
 \#\left\{j:\|\nabla\phi^j\|_{L_x^2}>\|\nabla W\|_{L_x^2}\right\}
 \leq1.
\end{equation}
\end{lemma}

\begin{proof}
Lemma~\ref{GNsharpz} and the Pohozaev identities yield
\[
 \|f\|_{L_x^{2^*}}^{2^*}
 \leq \|\nabla W\|_{L_x^2}^2
 \bigl(\tfrac{\|\nabla f\|_{L_x^2} }{\|\nabla W\|_{L_x^2} }\bigr)^{2^*}.
\]
This together with  $\|\nabla f\|_{L_x^2}\le M<\kappa_d\|\nabla W\|_{L_x^2}$ imply
\[
 E(f)\geq  
 \Big[\frac{1}{2}-\frac{1}{2^*}
 \Big(\tfrac{M}{\|\nabla W\|_{L_x^2}}\Big)^{\frac4{d-2}}\Big] \|\nabla f\|_{L_x^2}^2.
\]
The bracket is positive by the definition of $\kappa_d$, and \eqref{eq:ceiling-coercivity} follows. By Lemma \ref{lem:linear-profiles}, the kinetic-energy decoupling associated with \eqref{eq:lpd} gives
\[
 \|\nabla\phi^j\|_{L_x^2}\leq M,
 \qquad
 \limsup_{n\to\infty}\|\nabla w_n^J\|_{L_x^2}\leq M.
\]
Applying \eqref{eq:ceiling-coercivity} to $e^{it_n^j\Delta}\phi^j$
and passing to the limit give
\[
 E(U^j)\geq c_M\|\nabla\phi^j\|_{L_x^2}^2.
\]
The energy expansion \eqref{eq:energy-decoupling} follows from
\eqref{eq:kinetic-decoupling}--\eqref{eq:potential-decoupling}.
{For $E(w_n^J)$,
choose $M<M'<\kappa_d\|\nabla W\|_{L_x^2}$ and apply
\eqref{eq:ceiling-coercivity} with $M'$ for all sufficiently large
$n$.}

We now prove \eqref{eq:at-most-one-supercritical-profile}. 
Similar to \eqref{eq:ceiling-coercivity}, the sharp Sobolev inequality gives
\begin{align}\label{concavity0}
   \frac{E(f)}{E(W)}\geq
 \psi\Big(\frac{\|\nabla f\|_{L_x^2}^2}
 {\|\nabla W\|_{L_x^2}^2}\Big),  
\end{align}
where  $\psi(s):=\frac d2s-\frac{d-2}{2}s^{\frac d{d-2}}$ satisfying 
\[
 \psi(1)=1,\qquad \psi(\kappa_d^2)=0,\qquad
 \psi''(s)=-\frac d{d-2}s^{\frac{4-d}{d-2}}<0\quad(s>0).
\]
Thus concavity on $[1,\kappa_d^2]$ yields
\begin{equation}\label{concavity1}
     \psi(s)\geq\frac{\kappa_d^2-s}{\kappa_d^2-1},\quad \forall s\in[1,\kappa_d^2].
\end{equation}
Suppose that 
$\|\nabla\phi^j\|_{L_x^2}>\|\nabla W\|_{L_x^2}$, $j=1,2$, and denote
\[
 s_j:=\frac{\|\nabla\phi^j\|_{L_x^2}^2}
 {\|\nabla W\|_{L_x^2}^2},\qquad j=1,2.
\]
By \eqref{eq:kinetic-decoupling}, we have
\[
 1<s_j<\kappa_d^2 \;(j=1,2),\qquad
 s_1+s_2\leq\frac{M^2}{\|\nabla W\|_{L_x^2}^2}<\kappa_d^2.
\]
Applying \eqref{concavity0} and \eqref{concavity1} to $e^{it_n^j\Delta}\phi^j$ and passing to the limit give
\[
 \frac{E(U^j)}{E(W)}\geq
 \frac{\kappa_d^2-s_j}{\kappa_d^2-1},\qquad j=1,2.
\]
Since $E(U^j)\geq0$ for every $j\le J$, and
$\liminf_{n }E(w_n^J)\geq0$, \eqref{eq:energy-decoupling} and $\limsup_{n }E(u_n)\leq E(W)$ then yield $E(U^1)+E(U^2)\leq E(W)$. Hence,
\[
 1\geq\frac{E(U^1)+E(U^2)}{E(W)}
 \geq\frac{2\kappa_d^2-s_1-s_2}{\kappa_d^2-1}
 >\frac{\kappa_d^2}{\kappa_d^2-1}>1.
\]
This contradiction proves \eqref{eq:at-most-one-supercritical-profile}. The lemma is proved.
\end{proof}

\subsection{{A Palais-Smale argument in the interior of the upper cone and proof of Theorem~\ref{thm:general-subthreshold}(ii)}}\label{subsec:Thmsubthreshold}

We now prove Theorem~\ref{thm:general-subthreshold}(ii). We first extract a minimal-energy almost periodic solution, and then exclude it by a localized virial argument.

\begin{proposition}[Palais--Smale] 
\label{prop:critical-element}
Fix $M<\kappa_d\|\nabla W\|_{L_x^2}$.  Let $\mathcal B_M$ be the collection
of maximal-lifespan solutions $q$ to \eqref{1.1} {with
$I_q=(T_{q,-},T_{q,+})\ni0$}
satisfying
\[
{\|\nabla q(0)\|_{L_x^2}>\|\nabla W\|_{L_x^2},\quad E(q)<E(W),\quad
 \sup_{t\in[0,T_{q,+})}\|\nabla q(t)\|_{L_x^2}\leq M.
}\]
Set
\begin{equation}\label{eq:critical-energy}
 E_c:=\inf_{q\in\mathcal B_M}E(q).
\end{equation}
If $\mathcal B_M\neq\varnothing$, then there exists a nonzero global
solution $v:\R\times\R^d\to\C$ with the following properties:
\begin{enumerate}
 \item[(i)] $E(v)=E_c\in(0,E(W))$,
 $\|\nabla v(t)\|_{L_x^2}>\|\nabla W\|_{L_x^2}$ for $t\in\R$, and
 $\sup\limits_{t\in\R}\|\nabla v(t)\|_{L_x^2}\leq M$;
 \item[(ii)] $v$ is almost periodic modulo {symmetry group $G$} with frequency scale
 $N$ and spatial center $x$, and $N(t)\geq1$ for all $t\in\R$;
 \item[(iii)] there exists $c_0>0$ such that
 \begin{equation}\label{eq:critical-K-gap}
 K(v(t))\leq-c_0,\; \forall\;t\in\R.
 \end{equation}
\end{enumerate}
\end{proposition}

\begin{proof}
\begin{enumerate}
\item[(i)--(ii)] We first extract a nonlinear profile at the minimal energy
and establish compactness of its orbit. We then construct $v$ as a limit
of suitably rescaled time translates of this profile.
The variational dichotomy 
gives, for every
$q\in\mathcal B_M$ and $t\in I_q$,
 \begin{equation}\label{eq:supercritical-potential-gap}
 \|q(t)\|_{L_x^{2^*}}^{2^*}
 >\|W\|_{L_x^{2^*}}^{2^*}.
 \end{equation}
Thus no $q\in\mathcal B_M$ can scatter in any time direction. Consequently, every $q$ in $\mathcal B_M$ has infinite scattering size in both time directions.

Since $\mathcal{B}_M$ is nonempty,  $E_c<E(W)$ by \eqref{eq:critical-energy}. {Applying Lemma~\ref{lem:ceiling-coercivity} to $q(0)$ for any $q\in\mathcal B_M$, we have
\[
 E(q)=E(q(0))\geq c_M\|\nabla q(0)\|_{L_x^2}^2
 > c_M\|\nabla W\|_{L_x^2}^2>0.
\]
Taking the infimum on $B_M$ gives 
$$
E_c\ge c_M\|\nabla W\|_{L_x^2}^2>0.
$$}
We choose $q_n\in\mathcal B_M$ such that $E(q_n)\to E_c$. {Since every
$q_n$ has infinite forward scattering size, we may choose
$s_n\in(0,T_{q_n,+})$ such that}
\begin{equation}\label{eq:central-times}
{ S_{[0,s_n]}(q_n)\geq n,
 \qquad
 S_{[s_n,T_{q_n,+})}(q_n)=\infty.
}\end{equation}
Apply Lemma~\ref{lem:linear-profiles} to $q_n(s_n)$ and let $U^j$ be
the associated nonlinear profiles. Lemma~\ref{lem:ceiling-coercivity}
gives
\begin{equation}\label{eq:critical-energy-sum}
 E_c=\sum_{j=1}^J E(U^j)+E(w_n^J)+o_n(1), \quad E(U^j)\geq0,\;\liminf_{n\to\infty}E(w_n^J)\geq0.
\end{equation}

Set
\[
 \widetilde q_n^J(t):=\sum_{j=1}^J U_n^j(t)+e^{it\Delta}w_n^J.
\]
Suppose first that
\begin{equation}\label{eq:all-profiles-subcritical}
 \|\nabla\phi^j\|_{L_x^2}\leq\|\nabla W\|_{L_x^2}
 \qquad(\forall j\geq1).
\end{equation}
Since $0\le E(U^j)\leq E_c<E(W)$, the inequalities in
\eqref{eq:all-profiles-subcritical} are strict. Theorem~\ref{thm:KM}(i)
shows that every $U^j$ is global and scatters in both time directions. Similar to \cite{kenig-merle1}, 
energy decoupling, small-data theory, parameter orthogonality, and
\eqref{eq:linear-remainder-small} give
\begin{align*}
 &\sup_{J\geq J_0}\limsup_{n\to\infty}
 \left[S_{\R}(\widetilde q_n^J)
 +\|\widetilde q_n^J\|_{\dot S^1(\R)}\right]<\infty,\\
 &\lim_{J\to\infty}\limsup_{n\to\infty}
 \|\widetilde q_n^J(0)-q_n(s_n)\|_{\dot H^1}=0,\\
 &\lim_{J\to\infty}\limsup_{n\to\infty}
 \left\|(i\partial_t+\Delta)\widetilde q_n^J
 +|\widetilde q_n^J|^{\frac4{d-2}}\widetilde q_n^J
 \right\|_{\dot N^1(\R)}=0.
\end{align*}
 Together with Lemma~\ref{pertu}, 
 these yield
\[
 \sup_n S_{I_{q_n}}(q_n)<\infty,
\]
contradicting \eqref{eq:central-times}.  Hence, after relabeling, we may assume
\begin{equation}\label{eq:supercritical-main-profile}
 \|\nabla\phi^1\|_{L_x^2}>\|\nabla W\|_{L_x^2}.
\end{equation}

By \eqref{eq:at-most-one-supercritical-profile}, $\phi^1$ is the only
profile satisfying \eqref{eq:supercritical-main-profile}. For $j\geq2$,
we have $\|\nabla\phi^j\|_{L_x^2}<\|\nabla W\|_{L_x^2}$ and $U^j$ is global and scatters in both time directions as before.
In addition, $t_n^1\equiv0$. Otherwise, $t_n^1\to\pm\infty$ leads to
\[
 E(U^1)=\frac12\|\nabla\phi^1\|_{L_x^2}^2
 >\frac12\|\nabla W\|_{L_x^2}^2>E(W),
\]
contradicting \eqref{eq:critical-energy-sum}.

Let {$J_1=[a,b]\subset I_{U^1}$} be a compact interval containing zero and set
$J_{1,n}:=(\lambda_n^1)^2J_1$. The profile $U^1$ has finite scattering
norm on $J_1$, while all the other nonlinear profiles scatter on
$\R$. {We first claim that}
\[
{ s_n+(\lambda_n^1)^2a\geq0
 \qquad\text{for all sufficiently large }n.
}\]
{Otherwise, after passing to a subsequence, $a<0$ and}
\[
{ \left[-\frac{s_n}{(\lambda_n^1)^2},0\right]\subset[a,0].
}\]
{Stability on the translated interval $[-s_n,0]$ then gives a uniform
bound for $S_{[0,s_n]}(q_n)$, contradicting \eqref{eq:central-times}.
Stability on $J_{1,n}$ and the maximality of $I_{q_n}$ also give
$s_n+(\lambda_n^1)^2b<T_{q_n,+}$ for all sufficiently large $n$.
Thus $s_n+J_{1,n}\subset[0,T_{q_n,+})$ for all sufficiently large $n$, and}
\begin{equation}\label{eq:local-nonlinear-profile-approximation}
 \lim_{J\to\infty}\limsup_{n\to\infty}
 \|q_n(s_n+\cdot)-\widetilde q_n^J\|_{
 L_t^\infty\dot H_x^1(J_{1,n}\times\R^d)}=0.
\end{equation}
For each $t\in J_1$, kinetic-energy decoupling for the nonlinear
profiles (cf. \cite[Lemma 3.3]{KV1}) and \eqref{eq:local-nonlinear-profile-approximation} imply that
\begin{equation}\label{eq:evolved-kinetic-decoupling}
 \|\nabla U^1(t)\|_{L_x^2}
 \leq\liminf_{n\to\infty}
 \|\nabla q_n(s_n+(\lambda_n^1)^2t)\|_{L_x^2}
 \leq M.
\end{equation}
Since $J_1\Subset I_{U^1}$ is arbitrary, the same bound holds throughout
$I_{U^1}$. The variational dichotomy, $U^1(0)=\phi^1$, \eqref{eq:supercritical-main-profile} and the definition of $E_c$ give 
$$U^1\in\mathcal B_M, \quad E(U^1)\geq E_c.$$
Together with \eqref{eq:critical-energy-sum} and $t_n^1=0$, this yields
\[
E(U^1)=E(\phi^1)=E_c,\quad E(U^j)=0\;(j\geq2),\quad
 \lim_{J\to\infty}\limsup_{n\to\infty}E(w_n^J)=0.
\]
From coercivity \eqref{eq:ceiling-coercivity}, the original profile decomposition reduces to
\begin{equation}\label{eq:one-profile-reduction}
 q_n(s_n)=g_n^1\phi^1+o_n(1)
 \quad\hbox{in }\dot H^1(\R^d).
\end{equation}
Since $U^1\in\mathcal B_M$ and \eqref{eq:supercritical-potential-gap}, it blows up in both time directions.

We next use the minimality of $E_c$ to prove that $U^1$ is almost
periodic modulo {symmetry group $G$}. Let $\{\tau_n\}\subset I_{U^1}$ be arbitrary.
Since $U^1\in\mathcal B_M$ and $E(U^1)=E_c$, we have
\[
 S_{(\inf I_{U^1},\tau_n]}(U^1)
 =S_{[\tau_n,\sup I_{U^1})}(U^1)=\infty
 \qquad\text{for every }n.
\]
{Applying the same profile-decomposition argument to the sequence \(\{U^1(\tau_n)\}\), we obtain a nonlinear profile \(V\) with initial data in the interior of the upper cone, satisfying
\[
E(V)\leq E(U^1)=E_c<E(W).
\]
As in the preceding argument, the associated time-shift parameters can be taken to be zero.}
As in
\eqref{eq:evolved-kinetic-decoupling}, stability and kinetic-energy
decoupling give the bound $\|\nabla V(t)\|_{L_x^2}\leq M$ throughout its
maximal lifespan, so $V\in\mathcal B_M$. By the definition of $E_c$,
\[
 E_c\leq E(V)\leq E(U^1)=E_c.
\]
From energy decoupling, and the nonnegativity of the other profile energies
and the limiting remainder energy,  all other profiles have zero energy, and the remainder
energy tends to zero. The coercivity in Lemma~\ref{lem:ceiling-coercivity}
then gives, after passing to a subsequence, group elements $g_n\in G$
such that
\[
 g_nU^1(\tau_n)\to V(0)
 \quad\hbox{strongly in }\dot H^1(\R^d).
\]
Since $\{\tau_n\}$ was arbitrary, the orbit of $U^1$ is {precompact
modulo $G$}. Thus $U^1$ is almost periodic, and
Lemma~\ref{lem:finite-ap} gives $I_{U^1}=\R$.

We apply the normalization procedure from \cite{KV1}, following
\cite[Proof of Proposition~4.1]{MMMZ25}. This construction selects times
$r_n\in\R$ and group elements
$\widetilde g_n=g_{\widetilde\lambda_n,\widetilde x_n}\in G$ such that
\[
 \widetilde g_nU^1(r_n)\to v(0)
 \quad\hbox{strongly in }\dot H^1(\R^d),
\]
where $v:I_v\times\R^d\to\C$ is a nonzero maximal-lifespan
almost periodic solution. Lemma~\ref{lem:finite-ap} gives $I_v=\R$,
excluding the finite-lifespan alternative in this construction. In the
remaining alternatives, the frequency scale of the limiting solution
is bounded away from zero. Including a fixed scaling in
$\widetilde g_n$, we therefore obtain
\[
 N(t)\geq1\qquad(t\in\R).
\]
Moreover, stability gives, for every compact interval $J\subset\R$,
\[
 \lim_{n\to\infty}\sup_{t\in J}
 \bigl\|\widetilde g_nU^1(r_n+t/\widetilde\lambda_n^2)-v(t)
 \bigr\|_{\dot H^1}=0.
\]
Energy conservation and the invariance of the energy and the
$\dot H^1$ norm under $G$ therefore imply
\[
 E(v)=E_c,\qquad
 \|\nabla W\|_{L_x^2}\leq\|\nabla v(t)\|_{L_x^2}\leq M
 \qquad(t\in\R).
\]
Equality in the first kinetic inequality would imply $E(v)\geq E(W)$
by the sharp Sobolev inequality, contradicting $E_c<E(W)$.
Thus $v\in\mathcal B_M$, and items (i) and (ii) follow.

\item[(iii)] Finally, the non-scattering solution $v$ remains in $\mathcal B_M$.
Hence $\|\nabla v(t)\|_{L_x^2}>\|\nabla W\|_{L_x^2}$ for all $t$.  Since
$E(v)=E_c<E(W)$, the sharp Sobolev inequality and the usual energy
trapping argument give a constant $c_0=c_0(E_c,W)>0$ such that
\[
 K(v(t))\leq-c_0\qquad(t\in\R).
\]
\end{enumerate}
This proves the proposition.
\end{proof}

\begin{proposition}[Rigidity]\label{prop:ap-rigidity}
Let $v$ be the solution obtained in
Proposition~\ref{prop:critical-element}.  Then no such solution exists.
\end{proposition}

\begin{proof}
Let $M$, $E_c$, and $\mathcal B_M$ be as in
Proposition~\ref{prop:critical-element}, and set
$c_0:=-\sup_{t\in\R}K(v(t))>0$. Following \cite[Appendix~A]{MMMZ25} and
\cite[Sections~6 and 7]{KV1}, for some $\varepsilon_1>0$ we have
\begin{align}\label{Neone}
 v\in L_t^\infty\dot H_x^{-\varepsilon_1}
 \cap L_t^\infty L_x^2,
 \qquad N(t)\equiv1.
\end{align}
Define the mass and momentum of $v$ by
\[
 M(v):=\int_{\R^d}|v(t,x)|^2\,dx,
 \qquad
 P(v):=2\Im\int_{\R^d}\overline{v(t,x)}\nabla v(t,x)\,dx.
\]
Both quantities are finite and conserved.  In addition, $P(v)=0$.
If not, let
\[
 \xi_0:=-\frac{P(v)}{2M(v)},\qquad
 \widetilde v(t,x):=e^{ix\cdot\xi_0-it|\xi_0|^2}
 v(t,x-2t\xi_0).
\]
Then
 \begin{equation}\label{eq:critical-galilean}
 \|\nabla\widetilde v(t)\|_{L_x^2}^2
 =\|\nabla v(t)\|_{L_x^2}^2-\frac{|P(v)|^2}{4M(v)},
 \quad
 E(\widetilde v)=E(v)-\frac{|P(v)|^2}{8M(v)}<E_c.
\end{equation}
As the Galilean transform preserves the scattering size,
$\widetilde v$ does not scatter in either direction.  By
Theorem~\ref{thm:KM}(i), it remains {in the interior of the upper cone}.
Equation \eqref{eq:critical-galilean} then gives
$\widetilde v\in\mathcal B_M$, contradicting the definition of $E_c$.
Thus $P(v)=0$.

From \eqref{Neone} and almost periodicity, for every
$\eta>0$, there exists a constant $C(\eta)>0$ such that
\begin{equation}\label{eq:critical-L2-compactness}
 \sup_{t\in\R}\int_{|x-x(t)|>C(\eta)}
 \bigl(|\nabla v(t,x)|^2+|v(t,x)|^2\bigr)\,dx<\eta.
\end{equation}
Together with the zero momentum $P(v)=0$, this yields\footnote{The proof is similar to \cite{DHR08,KV1} and thus omitted.}
\begin{equation}\label{eq:critical-center-sublinear}
 x(t)=o(t)\qquad(t\to\infty).
\end{equation}

Note that $|M_R(f)|\lesssim R\|\nabla f\|_{L_x^2}\|f\|_{L_x^2}$, for some $C_v>0$, we have
\[
 \sup_{t\in\R}|M_R(v(t))|\le C_v R.
\]
  Recall that $K(v(t))\le -c_0$.
Integrating \eqref{eq:localized-virial} on $[0,T]$ with $T>0$ gives
\begin{equation}\label{eq:highd-integrated-virial}
 8c_0T
 \leq C_v R+\int_0^T|F_R(v(t))|\,dt.
\end{equation}
Let $\varepsilon,\eta>0$ and choose $R=\varepsilon T$. By \eqref{eq:critical-center-sublinear}, there exists
$T_{\varepsilon,\eta}$ such that
\[
 \varepsilon t-|x(t)|\geq C(\eta),
 \quad \forall t\geq T_{\varepsilon,\eta}.
\]
Together with almost periodicity and \eqref{eq:virial-tail-L2}, this yields
\[
\int_0^T|F_R(v(t))|\,dt=\int_0^{T_{\varepsilon,\eta}}+\int_ {T_{\varepsilon,\eta}}^{T} \; |F_R(v(t))|\,dt
 \lesssim_v T_{\varepsilon,\eta}
 +\left(\eta+\eta^{\frac{d-2}{d}}\right)T.
\]
Inserting this into \eqref{eq:highd-integrated-virial}, dividing by
$T$, and letting $T\to\infty$ gives
\[
 8c_0\lesssim_v\varepsilon+\eta+\eta^{\frac{d-2}{d}}.
\]
We get a contradiction by taking $\varepsilon,\eta$ sufficiently small. This completes the proof of Proposition \ref{prop:ap-rigidity}.
\end{proof}

We are now ready to prove Theorem~\ref{thm:general-subthreshold}(ii).
\begin{proof}[{\bf Proof of Theorem~\ref{thm:general-subthreshold}(ii)}]
{By time reversal, it suffices to prove the assertion at $T_+$.
Suppose otherwise that}
\[
{ \limsup_{t\uparrow T_+}\|\nabla u(t)\|_{L_x^2}
 <\kappa_d\|\nabla W\|_{L_x^2}.
}\]
{Choose $t_0\in I$ and $M$ such that}
\[
{ \sup_{t\in[t_0,T_+)}\|\nabla u(t)\|_{L_x^2}\leq M
 <\kappa_d\|\nabla W\|_{L_x^2}.
}\]
{The time-translated solution $q(t):=u(t_0+t)$ belongs to
$\mathcal B_M$. Proposition~\ref{prop:critical-element} therefore}
produces the solution $v$, contradicting
Proposition~\ref{prop:ap-rigidity}. {Applying the forward assertion to
$\overline{u(-t)}$ gives the assertion at $T_-$. This completes the proof.}
\end{proof}

\subsection{Compactness on the upper threshold sheet}\label{subsec:Thmthreshold}
In this subsection, we prove Theorem~\ref{thm:general-threshold}(ii). We first establish sequential convergence to the ground-state manifold.
\begin{proposition}[Sequential convergence]\label{prop:sequential}
Let $u$ be a solution to \eqref{1.1} satisfying \eqref{hyp:bounded} and \eqref{hyp:ceiling}. Then $T_+=\infty$ and there exists a sequence $t_n\to\infty$ such that
\begin{equation}\label{eq:sequential}
 \delta(u(t_n))\to0.
\end{equation}
\end{proposition}

We first establish global existence and compactness of the forward orbit and then prove Proposition~\ref{prop:sequential}.


\begin{proposition}[Global existence and compactness]\label{prop:forward-compactness}
Let $u$ be a solution to \eqref{1.1} satisfying \eqref{hyp:bounded} and \eqref{hyp:ceiling}.
Then $T_+=\infty$, and $u$ is almost periodic {modulo $G$}
on $[0,\infty)$.
\end{proposition}

\begin{proof}
Let $\{t_n\}\subset[0,T_+)$ be arbitrary. If, after passing to a
subsequence, $t_n\to t_*\in[0,T_+)$, then continuity gives
$u(t_n)\to u(t_*)$ in $\dot H^1(\R^d)$. Thus we may assume that
$t_n\to T_+$. By Theorem~\ref{T:local} and
Lemma~\ref{lem:threshold-branch},
\begin{equation}\label{eq:threshold-central-times}
 S_{[0,t_n]}(u)\to\infty,
 \qquad S_{[t_n,T_+)}(u)=\infty.
\end{equation}
Apply Lemma~\ref{lem:linear-profiles} to $u(t_n)$ and let $U^j$ be
the associated nonlinear profiles. Lemma~\ref{lem:ceiling-coercivity}
gives
\begin{equation}\label{eq:threshold-energy-sum}
 E(W)=\sum_{j=1}^J E(U^j)+E(w_n^J)+o_n(1),\quad
 E(U^j)\geq0,\;\liminf_{n\to\infty}E(w_n^J)\geq0.
\end{equation}

Set
\[
 \widetilde u_n^J(t):=\sum_{j=1}^J U_n^j(t)+e^{it\Delta}w_n^J.
\]
Suppose first that $E(U^j)<E(W)$ for every $j$. If
\begin{equation}\label{eq:threshold-all-profiles-subcritical}
 \|\nabla\phi^j\|_{L_x^2}\leq\|\nabla W\|_{L_x^2}
 \qquad(\forall j\geq1),
\end{equation}
then the inequalities in
\eqref{eq:threshold-all-profiles-subcritical} are strict.
Theorem~\ref{thm:KM}(i) shows that every $U^j$ is global and scatters
in both time directions. Similar to \cite{kenig-merle1}, energy
decoupling, small-data theory, parameter orthogonality, and
\eqref{eq:linear-remainder-small} give
\begin{align*}
 &\sup_{J\geq J_0}\limsup_{n\to\infty}
 \left[S_{\R}(\widetilde u_n^J)
 +\|\widetilde u_n^J\|_{\dot S^1(\R)}\right]<\infty,\\
 &\lim_{J\to\infty}\limsup_{n\to\infty}
 \|\widetilde u_n^J(0)-u(t_n)\|_{\dot H^1}=0,\\
 &\lim_{J\to\infty}\limsup_{n\to\infty}
 \left\|(i\partial_t+\Delta)\widetilde u_n^J
 +|\widetilde u_n^J|^{\frac4{d-2}}\widetilde u_n^J
 \right\|_{\dot N^1(\R)}=0.
\end{align*}
Together with Lemma~\ref{pertu}, these yield
\[
 S_{[0,T_+)}(u)<\infty,
\]
contradicting \eqref{eq:threshold-central-times}. Hence, after
relabeling, we may assume
\begin{equation}\label{eq:threshold-supercritical-main-profile}
 \|\nabla\phi^1\|_{L_x^2}>\|\nabla W\|_{L_x^2}.
\end{equation}

By \eqref{eq:at-most-one-supercritical-profile}, $\phi^1$ is the only
profile satisfying \eqref{eq:threshold-supercritical-main-profile}.
For $j\geq2$, we have
$\|\nabla\phi^j\|_{L_x^2}<\|\nabla W\|_{L_x^2}$ and $U^j$ is global
and scatters in both time directions as before. In addition,
$t_n^1\equiv0$. Otherwise, $t_n^1\to\pm\infty$ leads to
\[
 E(U^1)=\frac12\|\nabla\phi^1\|_{L_x^2}^2
 >\frac12\|\nabla W\|_{L_x^2}^2>E(W),
\]
contradicting \eqref{eq:threshold-energy-sum}.

Let $J_1=[a,b]\subset I_{U^1}$ be a compact interval containing zero
and set $J_{1,n}:=(\lambda_n^1)^2J_1$. The profile $U^1$ has finite
scattering norm on $J_1$, while all the other nonlinear profiles
scatter on $\R$. We first show that
\[
 t_n+(\lambda_n^1)^2a\geq0
 \qquad\text{for all sufficiently large }n.
\]
If this fails, then $a<0$ and, after passing to a subsequence,
\[
 0\leq\frac{t_n}{(\lambda_n^1)^2}<-a,\qquad
 \left[-\frac{t_n}{(\lambda_n^1)^2},0\right]\subset[a,0].
\]

Stability on $[-t_n,0]$ then gives a uniform bound for
$S_{[0,t_n]}(u)$, contradicting
\eqref{eq:threshold-central-times}.
A similar argument shows that
$t_n+(\lambda_n^1)^2b<T_+$ for all sufficiently large $n$.
Hence
$t_n+J_{1,n}\subset[0,T_+)$
for all sufficiently large $n$, and the nonlinear-profile
approximation yields
\begin{equation}\label{eq:threshold-local-profile-approximation}
 \lim_{J\to\infty}\limsup_{n\to\infty}
 \|u(t_n+\cdot)-\widetilde u_n^J\|_{
 L_t^\infty\dot H_x^1(J_{1,n}\times\R^d)}=0.
\end{equation}
For each $t\in J_1$, kinetic-energy decoupling for the nonlinear
profiles (cf. \cite[Lemma 3.3]{KV1}) and
\eqref{eq:threshold-local-profile-approximation} imply that
\begin{equation}\label{eq:threshold-evolved-kinetic-decoupling}
 \|\nabla U^1(t)\|_{L_x^2}
 \leq\liminf_{n\to\infty}
 \|\nabla u(t_n+(\lambda_n^1)^2t)\|_{L_x^2}
 \leq M.
\end{equation}
Since $J_1\Subset I_{U^1}$ is arbitrary, the same bound holds
throughout $I_{U^1}$. The variational dichotomy, $U^1(0)=\phi^1$,
and \eqref{eq:threshold-supercritical-main-profile} give
\[
 E(U^1)<E(W),\qquad
 \|\nabla U^1(t)\|_{L_x^2}>\|\nabla W\|_{L_x^2}
 \qquad(t\in I_{U^1}).
\]
This contradicts Theorem~\ref{thm:general-subthreshold}(ii).

Hence, after relabeling, $E(U^1)=E(W)$. Together with
\eqref{eq:threshold-energy-sum}, this yields
\[
 E(U^j)=0\;(j\geq2),\qquad
 \lim_{J\to\infty}\limsup_{n\to\infty}E(w_n^J)=0.
\]
From coercivity \eqref{eq:ceiling-coercivity}, the original profile
decomposition reduces to
\begin{equation}\label{eq:threshold-one-profile-reduction}
 u(t_n)=g_n^1e^{it_n^1\Delta}\phi^1+o_n(1)
 \quad\hbox{in }\dot H^1(\R^d).
\end{equation}
Since $\|\nabla u(t_n)\|_{L_x^2}>\|\nabla W\|_{L_x^2}$, we have
$\|\nabla\phi^1\|_{L_x^2}\geq\|\nabla W\|_{L_x^2}$.
In addition, $t_n^1\equiv0$. Otherwise, $t_n^1\to\pm\infty$ leads to
\[
 E(U^1)=\frac12\|\nabla\phi^1\|_{L_x^2}^2
 \geq\frac12\|\nabla W\|_{L_x^2}^2>E(W),
\]
contradicting \eqref{eq:threshold-energy-sum}. Consequently,
\[
 (g_n^1)^{-1}u(t_n)\to\phi^1
 \quad\hbox{strongly in }\dot H^1(\R^d).
\]
Since $\{t_n\}$ was arbitrary, the forward orbit of $u$ is precompact
modulo symmetries. Thus $u$ is almost periodic on $[0,T_+)$, and
Lemma~\ref{lem:finite-ap} gives $T_+=\infty$. This proves the proposition.
\end{proof}

\begin{proof}[{\bf Proof of Proposition~\ref{prop:sequential}}]
Suppose to the contrary that there exists $\delta_0>0$ such that
\begin{equation}\label{eq:delta-gap}
 \delta(u(t))\geq\delta_0\qquad(t\geq0).
\end{equation}
By Proposition~\ref{prop:forward-compactness}, $u$ is almost periodic
on $[0,\infty)$.  Choose $\tau_n\to\infty$.  After applying the
compactness parameters and passing to a subsequence, there are
group elements $g_n=g_{\lambda_n,x_n}\in G$ and
$v_0\in\dot H^1(\mathbb{R}^d)$ such that
\begin{equation}\label{eq:threshold-limit-data}
 g_nu(\tau_n)\to v_0
 \quad\hbox{strongly in }\dot H^1(\R^d).
\end{equation}
Let $v:I_v\times\R^d\to\C$ be the maximal-lifespan solution with
$v(0)=v_0$. By Lemma~\ref{lem:threshold-branch},
\[
 S_{[0,\tau_n]}(u)\to\infty,\qquad
 S_{[\tau_n,\infty)}(u)=\infty.
\]
Together with stability, these imply that every compact subinterval of
$I_v$ is contained in $(-\lambda_n^2\tau_n,\infty)$ for all sufficiently
large $n$. Proposition~\ref{prop:forward-compactness} and stability then
show that $v$ is almost periodic on $I_v$. Lemma~\ref{lem:finite-ap}
gives $I_v=\R$. Strong convergence on compact time intervals also yields
$E(v)=E(W)$ and $\sup_{s\in\R}\|\nabla v(s)\|_{L_x^2}\leq M$.
For every $s\in\R$, \eqref{eq:delta-gap} gives
\begin{equation}\label{eq:gap-passes-to-limit}
 \delta(v(s))\geq\delta_0,\qquad
 K(v(s))=-\frac2{d-2}\delta(v(s))
 \leq-\frac{2\delta_0}{d-2}
 \qquad(s\in\R).
\end{equation}
The variational dichotomy gives blow-up in both time directions. As in
the proof of Proposition~\ref{prop:critical-element}, the normalization
construction of \cite{KV1} (see also
\cite[Proof of Proposition~4.1]{MMMZ25}) selects times $r_n\in\R$ and
group elements $\widetilde g_n=g_{\widetilde\lambda_n,\widetilde x_n}$
such that
\[
 \widetilde g_nv(r_n)\to w(0)
 \quad\hbox{strongly in }\dot H^1(\R^d),
\]
where $w:I_w\times\R^d\to\C$ is a nonzero maximal-lifespan
almost periodic solution. Lemma~\ref{lem:finite-ap} gives $I_w=\R$,
excluding the finite-lifespan alternative. The remaining alternatives
have frequency scale bounded away from zero. Including a fixed scaling
in $\widetilde g_n$, we obtain
\[
 N(t)\geq1\qquad(t\in\R).
\]
For every compact interval $J\subset\R$, stability gives
\[
 \lim_{n\to\infty}\sup_{t\in J}
 \bigl\|\widetilde g_nv(r_n+t/\widetilde\lambda_n^2)-w(t)
 \bigr\|_{\dot H^1}=0.
\]
Consequently, the energy, the kinetic bound, and
\eqref{eq:gap-passes-to-limit} pass to $w$:
\[
 E(w)=E(W),\qquad
 \sup_{t\in\R}\|\nabla w(t)\|_{L_x^2}\leq M,\qquad
 K(w(t))\leq-\frac{2\delta_0}{d-2}\quad(t\in\R).
\]
In particular, $w$ lies on the {upper sheet of the threshold surface} and blows
up in both time directions.

Following \cite[Appendix~A]{MMMZ25} and
\cite[Sections~6 and 7]{KV1}, for some $\varepsilon_1>0$ we have
\[
 w\in L_t^\infty\dot H_x^{-\varepsilon_1}
 \cap L_t^\infty L_x^2,
 \qquad N(t)\equiv1.
\]
Define $M(w)$ and $P(w)$ as in the proof of
Proposition~\ref{prop:ap-rigidity}.  We claim that $P(w)=0$.  If not,
set
\[
 \xi_0:=-\frac{P(w)}{2M(w)},\qquad
 \widetilde w(t,x):=e^{ix\cdot\xi_0-it|\xi_0|^2}
 w(t,x-2t\xi_0).
\]
Then
\[
 \|\nabla\widetilde w(t)\|_{L_x^2}^2
 =\|\nabla w(t)\|_{L_x^2}^2-\frac{|P(w)|^2}{4M(w)},
 \qquad
 E(\widetilde w)=E(W)-\frac{|P(w)|^2}{8M(w)}<E(W).
\]
The solution $\widetilde w$ blows up in both time directions.  By
Theorem~\ref{thm:KM}(i), it lies {in the interior of the upper cone}.
It is therefore excluded by
Theorem~\ref{thm:general-subthreshold}(ii).  Hence $P(w)=0$.

Negative regularity and almost periodicity give the spatial tightness estimate \eqref{eq:critical-L2-compactness}, with $w$ in place of $v$.  The truncated-position
argument in the proof of Proposition~\ref{prop:ap-rigidity} then gives
\[
 x(t)=o(t)\qquad(t\to\infty).
\]
Note that $|M_R(f)|\lesssim R\|\nabla f\|_{L_x^2}\|f\|_{L_x^2}$, we have
\[
 \sup_{t\in\R}|M_R(w(t))|\lesssim_w R.
\]
Let $R=\varepsilon T$.  Integrating the localized virial identity on
$[0,T]$ and using \eqref{eq:gap-passes-to-limit} gives
\[
 \frac{16\delta_0}{d-2}T
 \leq C_w\varepsilon T+\int_0^T|F_R(w(t))|\,dt.
\]
As in the proof of Proposition~\ref{prop:ap-rigidity}, almost
periodicity and $x(t)=o(t)$ imply
\[
 \int_0^T|F_R(w(t))|\,dt
 \lesssim_w T_{\varepsilon,\eta}
 +\left(\eta+\eta^{\frac{d-2}{d}}\right)T.
\]
Dividing by $T$, letting $T\to\infty$, and then choosing
$\varepsilon,\eta$ small yields a contradiction.  This proves
\eqref{eq:sequential}.
\end{proof}

We now complete the proof of Theorem~\ref{thm:general-threshold}(ii). We begin with a lemma that controls the variation of $\delta$ on an interval by the scattering size of that interval.

\begin{lemma}\label{lem:delta-variation}
Let $u$ satisfy the assumptions of Theorem \ref{thm:general-threshold}(ii), and let $J\subset[0,\infty)$ be
an interval such that $S_J(u)\leq L$.  There exist $c(L,M)>0$ and
$C(L,M)\geq1$ such that
\begin{equation}\label{eq:delta-variation}
\inf_{t\in J}\delta(u(t))\leq c(L,M)
\quad\Longrightarrow\quad
\sup_{t\in J}\delta(u(t))\leq C(L,M)\inf_{t\in J}\delta(u(t)).
\end{equation}
\end{lemma}

\begin{proof}
Put $p:=2(d+2)/(d-2)$. We first observe that, whenever
$0<\delta(u(t))<\delta_0$,
\begin{equation}\label{eq:delta-variation-scale-size}
 \|u(t)\|_{L_x^p}^p\gtrsim\lambda(t)^2
\end{equation}
for almost every $t$. Indeed, by Lemma~\ref{modula} and a sufficiently
small choice of $\delta_0$, the normalized function
$u_{[\theta(t),\lambda(t),x(t)]}(t)=W+v(t)$ satisfies
\[
 \|W+v(t)\|_{L_x^2(1\leq|x|\leq2)}\geq c>0.
\]
H\"older's inequality on this fixed annulus and scaling give
\eqref{eq:delta-variation-scale-size}.

Choose $t_0\in J$ with
\[
 \delta(u(t_0))\leq\inf_{t\in J}\delta(u(t))+\varepsilon,
\]
where $\varepsilon>0$ is sufficiently small. Let $J_0$ be the connected
component of $\{t\in J:\delta(u(t))<\delta_0\}$ containing $t_0$. On $J_0$,
\eqref{equivq1}, \eqref{Mod2}, and
\eqref{eq:delta-variation-scale-size} imply
\[
 \left|\frac{d}{dt}\log|\alpha(t)|\right|\lesssim\lambda(t)^2,
 \qquad
 \int_{J_0}\lambda(t)^2\,dt\lesssim S_J(u)\leq L.
\]
Here $\alpha$ does not vanish, since $|\alpha|\sim\delta(u)>0$.
Integrating in either time direction from $t_0$, we obtain
\[
 \delta(u(t))\leq C\exp(CL)\delta(u(t_0)).
\]
Choose $c(L,M)>0$ so small that the right-hand side is less than
$\delta_0/2$ whenever $\delta(u(t_0))\leq2c(L,M)$.
By continuity, the component is all of $J$. Letting
$\varepsilon\downarrow0$ proves the lemma.
\end{proof}

We now prove Theorem~\ref{thm:general-threshold}(ii).
\begin{proof}[{\bf Proof of Theorem~\ref{thm:general-threshold}(ii)}]
By Proposition~\ref{prop:sequential}, there exists an increasing
sequence $t_n\to\infty$ such that $\delta(u(t_n))\to0$.  We divide
the proof into two cases:
\begin{enumerate}
\item[(i)] There exists $t_0>0$ such that
\[
 \sup_{t\in[t_0,\infty)}\delta(u(t))\leq\eta_*.
\]
\item[(ii)] There exists an increasing sequence $t_n^-\to\infty$ such
that
\[
 \delta(u(t_n^-))>\eta_*.
\]
\end{enumerate}

In case (i), Corollary~\ref{t2.3}, after a time translation, gives
$u=W^+$ modulo symmetries.

Suppose that case (ii) holds.  Passing to a subsequence, we may assume
\[
 t_n^-<t_n<t_{n+1}^-,\qquad \delta(u(t_n))<\eta_*.
\]
Set
\[
 t_n^+:=\inf\left\{t\in[t_n^-,t_n]:
 \sup_{\tau\in[t,t_n]}\delta(u(\tau))<\eta_*\right\}.
\]
Then $t_n^+\to\infty$ and
\begin{equation}\label{eq:exit-level}
 \delta(u(t_n^+))=\eta_*.
\end{equation}
Lemma~\ref{lem:delta-variation} shows that
\begin{equation}\label{eq:large-transition-norm}
 S_{[t_n^+,t_n]}(u)\to\infty.
\end{equation}

By Proposition~\ref{prop:forward-compactness}, there exist symmetries $g_n$ and $w_0\in\dot H^1(\mathbb{R}^d)$ such that
\begin{equation}\label{eq:conditional-data}
 g_nu(t_n^+)\to w_0\quad\text{in }\dot H^1(\R^d),
 \qquad \delta(w_0)=\eta_*.
\end{equation}
Write $g_n=g_{\lambda_n,x_n}$ and let $w$ be the solution
with $w(0)=w_0$.  Define
\[
 u_n(t,x)
 :=\lambda_n^{-\frac{d-2}{2}}
 u\left(t_n^++\frac{t}{\lambda_n^2},
        \frac{x-x_n}{\lambda_n}\right).
\]
Then $u_n(0)\to w_0$ in $\dot H^1(\R^d)$.  Set
\[
 \sigma_n:=\lambda_n^2(t_n-t_n^+),
 \qquad \ell_n:=\lambda_n^2t_n^+.
\]
By scale invariance and \eqref{eq:large-transition-norm},
\begin{equation}\label{eq:rescaled-transition}
 S_{[0,\sigma_n]}(u_n)
 =S_{[t_n^+,t_n]}(u)\to\infty.
\end{equation}

By Lemma~\ref{lem:threshold-branch},
$S_{[0,t_n^+]}(u)\to\infty$. The divergence of the scattering sizes
on both sides, together with stability, implies that every compact
subinterval of $I_w$ is contained in $(-\ell_n,\sigma_n)$ for all
sufficiently large $n$. Otherwise, stability on a fixed compact
interval would bound one of these divergent scattering sizes.
Proposition~\ref{prop:forward-compactness} and stability therefore
show that $w$ is almost periodic and
\[
 \sup_{t\in I_w}\|\nabla w(t)\|_{L_x^2}\leq M.
\]
Lemma~\ref{lem:finite-ap} gives $I_w=\R$. The same compact-interval
coverage now implies
\[
 \sigma_n\to\infty,\qquad
 \ell_n\to\infty.
\]
Thus $w$ is a global almost-periodic {solution on the upper threshold sheet}
which blows up in both time directions.  The definition of $t_n^+$ gives
\[
 \delta(w(t))\leq\eta_*\qquad(t\geq0).
\]
Corollary~\ref{t2.3} shows that $w$ coincides with $W^+$ modulo a fixed
symmetry. Since $d\geq5$, $W^+$ blows up in finite negative time, as
recalled in the introduction. This contradicts $I_w=\R$. Thus case
(ii) is impossible.
{For the final assertion, the two endpoint bounds yield
\eqref{hyp:ceiling} after suitable time translations, respectively
for $u$ and $\overline{u(-t)}$. The preceding classification then gives
$I=\R$ and $u=W^+$ modulo symmetries, contradicting the finite backward
lifespan of $W^+$.}
This completes the proof.
\end{proof}

\appendix

\section{Finite variance of the special solution}\label{Wplus}

In this section, we show that the special solution $W^+$ has finite variance if and only if $d\geq7$.

\begin{proposition} Let $W^{+}:I\times\R^d\to\C$ be the solution to \eqref{1.1} as described in the introduction. Then
    $W^{+}(t)\in \Sigma$ for any $t\in I$ if and only if $d\geq7$.
\end{proposition}

\begin{proof}

We first prove $W^+(t)\notin \Sigma$ when $d\leq 6$. We argue by contradiction. Suppose this fails, then as $W^+$ solves \eqref{1.1}, $W^{+}(t)\in \Sigma$ for all $t\in I$. Moreover, by Lemma~\ref{unib} (i),
\[
\sup_{t\geq0}\|xW^+(t)\|_{L_x^2}<\infty.
\]
Mass conservation, the preceding variance bound, and convergence in
$\dot H^1(\R^d)$ imply that $\{W^+(t):t\geq0\}$ is bounded in $\Sigma$.
Since $\Sigma$ is a Hilbert space, there are $t_n\to\infty$ and
$\phi\in\Sigma$ such that
\[
 W^+(t_n)\rightharpoonup\phi\qquad\text{weakly in }\Sigma.
\]

We recall that there exists $c_0>0$ such that
\begin{equation}\label{app1}
    \|W^+-W\|_{\dot{H}^1}\lesssim e^{-c_0t} \qtq{for all}t\geq0.
\end{equation}
Therefore, $\phi=W$, which contradicts the fact that $W\notin\Sigma$.

It remains to prove that $W^+(t)\in\Sigma$ when $d\geq7$.  It suffices
to find $t_0\in I$ such that $W^+(t_0)\in\Sigma$.

Let $V_R$ and $M_R$ be as in
\eqref{eq:localized-virial-definitions}.  By
\eqref{eq:localized-variance},
\[
 \frac{d}{dt}V_R(W^+(t))=M_R(W^+(t)).
\]
By \eqref{app1} and Lemma~\ref{modula}, the modulation parameters
stay close to those of $W$ for $t\geq t_0$, with $t_0\gg1$ fixed.
Applying \eqref{VED} to the modulated ground
state and using \eqref{equivq1} and
\eqref{thlocalv}, we obtain
\[
\frac{d}{dt}M_R(W^+(t))
\leq-\frac{8}{d-2}\delta(W^+(t))
\]
uniformly for all sufficiently large $R$ and all $t\geq t_0$.
Moreover, \eqref{rac0} gives, for fixed $R$,
\[
M_R(W^+(t))\to 0\qtq{as}t\to\infty.
\]
Hence, by the fundamental theorem of calculus,
\[
M_R(W^+(t))
\geq \int_t^{\infty}\frac{8}{d-2}\delta(W^+(s))\,ds>0,
\qquad t\geq t_0.
\]
Thus $V_R(W^+(t))$ is strictly increasing on $[t_0,\infty)$ for all
$R\gg1$.  For fixed $R\gg1$,
\[
V_R(W^+(t_0))\leq\lim_{t\to\infty}V_R(W^+(t))=V_R(W).
\]
Since $W\in\Sigma$ for $d\geq7$, letting $R\to\infty$ yields
\[
\int_{\R^d}|xW^+(t_0,x)|^2\,dx
\leq\liminf_{R\to\infty}V_R(W^+(t_0))
\leq\int_{\R^d}|xW(x)|^2\,dx<\infty,
\]
and hence
$W^+(t_0)\in\Sigma$ as desired.
\end{proof}


\begin{thebibliography}{99}

\bibitem{Au76} T. Aubin. \emph{\'{E}quations diff\'erentielles non lin\'eaires et probl\`eme de {Y}amabe concernant la courbure scalaire}. J. Math. Pures Appl. (9), \textbf{55}(3):269--296, 1976.

\bibitem{Bonica} V. Banica, \emph{Remarks on the blow-up for the Schr\"odinger equation with critical mass on a plane domain}, Ann. Sc. Norm. Super. Pisa Cl. Sci. (5) \textbf{3} (2004), no. 1, 139--170.

\bibitem{CFR22} L. Campos, L. G. Farah, and S. Roudenko, \emph{Threshold solutions for the nonlinear Schr\"odinger equation.} Rev. Mat. Iberoam. \textbf{38} (2022), no. 5, 1637--1708.

{\bibitem{CH} P. Chung and D. Han, \emph{Threshold scattering for the three-dimensional focusing energy-critical nonlinear Schr\"odinger equation.} Preprint, arXiv: 2610.10658.}

\bibitem{dodson1} B. Dodson, \emph{Global well-posedness and scattering for the focusing, cubic Schr\"odinger equation in dimension $ d = 4 $.} Ann. Sci. Ecole Norm. S., serie 4, \textbf{52} (2019), 139--180.

\bibitem{DHR08} T. Duyckaerts, J. Holmer, and S. Roudenko, \emph{Scattering for the non-radial 3D cubic nonlinear Schr\"odinger equation.} Math. Res. Lett. \textbf{15} (2008), 1233--1250. 

\bibitem{DM08} T. Duyckaerts and F. Merle, \emph{Dynamic of threshold solutions for energy-critical NLS.}
Geom. Funct. Anal. \textbf{18} (2009), no. 6, 1787--1840.

\bibitem{DR10} T. Duyckaerts and S. Roudenko, \emph{Threshold solutions for the focusing 3D cubic Schr\"odinger equation.} Rev. Mat. Iberoam. \textbf{26} (2010), no.~1, 1--56.


\bibitem{GinibreVelo} J. Ginibre and G. Velo, \emph{Smoothing properties and retarded estimates for some dispersive evolution equations.} Comm. Math. Phys. \textbf{144} (1992), 163--188.

\bibitem{KeelTao}  M. Keel and T. Tao, \emph{Endpoint Strichartz estimates.} Amer. J. Math. \textbf{120} (1998), no. 5, 955--980.

\bibitem{kenig-merle1} C. E. Kenig and F. Merle, \emph{Global well-posedness, scattering and blow up for the energy-critical, focusing, nonlinear Schr\"odinger equation in the radial case.} Invent. Math. \textbf{166} (2006), 645--675.

\bibitem{Keraani} S. Keraani, \emph{On the defect of compactness for the Strichartz estimates of the Schr\"odinger equations.} J. Differential Equations \textbf{175} (2001), no. 2, 353--392.



\bibitem{KVnote} R. Killip and M. Vi\c{s}an, \emph{Nonlinear Schr\"odinger equations at critical regularity.} Lecture notes prepared for Clay Mathematics Institute Summer School, Z$\ddot{u}$rich, Switzerland, 2008.

\bibitem{KV1} R. Killip and M. Vi\c{s}an, \emph{The focusing energy-critical nonlinear Schr\"odinger equation in dimensions five and higher.} Amer. J. Math., \textbf{132} (2010), 361--424.

\bibitem{LZ09} D. Li and X. Zhang, \emph{Dynamics for the energy critical nonlinear Schr\"odinger equation in high dimensions.} J. Funct. Anal. \textbf{256} (2009),  1928--1961.

\bibitem{MMMZ25} Z. Ma, C. Miao, J. Murphy, and J. Zheng, \emph{Dynamics of subcritical threshold solutions for the 4d energy-critical NLS.} To appear in Anal. PDE.

\bibitem{MR08} F. Merle and P. Rapha\"el, \emph{Blow up of the critical norm for some radial $L^2$ super critical nonlinear Schr\"odinger equations.} Amer. J. Math. \textbf{130} (2008), no. 4, 945--978.


\bibitem{Strichartz} R. Strichartz, \emph{Restrictions of Fourier transforms to quadratic surfaces and decay of solutions of wave equations.} Duke Math. J. \textbf{44} (1977), no. 3, 705--714.

\bibitem{SZ} Q. Su and Z. Zhao, \emph{Scattering below the ground state for the three-dimensional focusing energy-critical NLS.} Preprint, arXiv: 2609.30228.

\bibitem{Ta76} G. Talenti, Best constant in {S}obolev inequality.  Ann. Mat. Pura Appl. (4), \textbf{110} (1976), 
353--372.

\bibitem{Tao1} T. Tao and M. Vi\c{s}an, \emph{Stability of energy-critical nonlinear Schr\"odinger equations in high dimensions.} Electron. J. Differential Equations, \textbf{118} (2005), 1--28.

\bibitem{VisanOberwolfach}  M. Vi\c{s}an, \emph{Dispersive Equations.} In Dispersive equations and nonlinear waves. Oberwolfach Seminars, 45. Birkh\"auser/Springer, Basel, 2014. xii+312 pp.

\end{thebibliography}
\end{document}